\documentclass[12pt,french,english]{article}
\usepackage{fourier}

\usepackage[T1]{fontenc}
\usepackage[utf8]{inputenc}
\usepackage{babel}
\usepackage{mathrsfs}
\usepackage{url}
\usepackage{amsmath}
\usepackage{amsthm}
\usepackage{amssymb}
\usepackage{geometry}
\usepackage{fancyhdr}
\usepackage{setspace}
\usepackage[bookmarks=true,bookmarksnumbered=true,bookmarksopen=true,bookmarksopenlevel=4,
 breaklinks=true,pdfborder={0 0 1},backref=section,colorlinks=false]
 {hyperref}

\makeatletter
\numberwithin{equation}{section}
\newcommand{\lyxaddress}[1]{
	\par {\raggedright #1
	\vspace{1.4em}
	\noindent\par}
}

\@ifundefined{date}{}{\date{}}

\usepackage{babel}
\usepackage{varioref}
\usepackage{float}
\usepackage{mathrsfs}
\usepackage{url}
\usepackage{amsthm}

\usepackage{fancyhdr}

\numberwithin{figure}{section}
\numberwithin{table}{section}
\numberwithin{equation}{section}

\newcommand{\p}{\ensuremath{\partial}}
\newcommand{\R}{\mathbb{R}}

\makeatother

\theoremstyle{plain}
\newtheorem{thm}{\protect\theoremname}[section]
\newtheorem{lem}[thm]{\protect\lemmaname}
\theoremstyle{remark}
\newtheorem{notation}[thm]{\protect\notationname}
\theoremstyle{definition}
\newtheorem{defn}[thm]{\protect\definitionname}
\theoremstyle{plain}
\newtheorem{prop}[thm]{\protect\propositionname}
\newtheorem{cor}[thm]{\protect\corollaryname}
\addto\captionsenglish{\renewcommand{\corollaryname}{Corollary}}
\addto\captionsenglish{\renewcommand{\definitionname}{Definition}}
\addto\captionsenglish{\renewcommand{\lemmaname}{Lemma}}
\addto\captionsenglish{\renewcommand{\notationname}{Notation}}
\addto\captionsenglish{\renewcommand{\propositionname}{Proposition}}
\addto\captionsenglish{\renewcommand{\theoremname}{Theorem}}

\providecommand{\corollaryname}{Corollary}
\providecommand{\definitionname}{Definition}
\providecommand{\lemmaname}{Lemma}
\providecommand{\notationname}{Notation}
\providecommand{\propositionname}{Proposition}
\providecommand{\theoremname}{Theorem}

\begin{document}
\title{Global classical solutions for the general planar Broadwell model
with arbitrary velocity orientation: a fixed-point framework }
\author{Koudzo Togbévi Selom, Sobah\textsuperscript{1,{*}} and Amah Séna,
d'Almeida\textsuperscript{2}}
\maketitle

\lyxaddress{1,2 : Department of Mathematics, University of Lomé, Togo\\
 {*} corresponding author: deselium@gmail.com}
\begin{abstract}
The present work extends a fixed-point framework previously developed
for the standard planar four-velocity Broadwell model to the general
four-velocity configuration with an arbitrary orientation of the discrete
velocities. This provides a non-axis-aligned test of the framework
and leads to a global-in-time existence and uniqueness result for
classical solutions of the associated initial-boundary value problem
in a rectangular domain. Under suitable assumptions on the initial
and boundary data, we establish the existence and uniqueness of a
global-in-time classical solution. The analysis identifies the geometric
and analytic properties of the fixed-point approach that remain valid
independently of the orientation of the discrete velocities. The results
demonstrate the applicability of the framework to the general planar
four-velocity Broadwell model beyond the standard axis-aligned configuration.
\end{abstract}


\section*{Introduction}

Initiated by Carleman \cite{carleman} and extensively developed by
Broadwell \cite{3,4} and Gatignol \cite{1} to approximate the Boltzmann
equation, discrete velocity models have seen robust global existence
results in one spatial dimension \cite{5,kawasima,tartar,Cabanne-kawashima,7,2,6,11,14,compte rendu meca},
whereas their non-stationary multidimensional counterpart remains
largely open.This paper stands as a direct analytical continuation
of the mathematical frameworks recently developed in \cite{Sob Alm}
and \cite{sob almeida 2026 arxiv 1} for multidimensional discrete
kinetic equations. In the study of the non-stationary multidimensional
four-velocity Broadwell model, establishing global-in-time classical
solutions remains a major challenge. Recently, a localized fixed-point
methodology was introduced in \cite{Sob Alm} to secure classical
solutions for the general four-velocity planar Broadwell model, though
the results were restricted to a local time horizon. Parallelly, the
authors achieved a global-in-time extension in \cite{sob almeida 2026 arxiv 1},
but the analysis was strictly confined to a standard axis-aligned
velocity layout in a rectangular domain.

The primary objective of the present work is to bridge the gap between
these two recent advancements. Rather than restricting the mathematical
analysis to a specific coordinate configuration, we demonstrate that
the fixed-point framework and the a priori estimation techniques established
in \cite{sob almeida 2026 arxiv 1} can be fully generalized. We test
the limits of this methodology by applying it directly to the model
introduced in \cite{Sob Alm}, which features an arbitrary orientation
angle $\theta\in]0,\pi/2[$.

We prove that despite the significant geometric complexities of the
model, the fixed-point methodology remains exceptionally robust. We
establish that the same formalism used in \cite{sob almeida 2026 arxiv 1}
can still be used to solve the actual problem to establish the global-in-time
existence, uniqueness, and non-negativity of classical solutions under
a meticulously computed threshold condition on the initial-boundary
norms and their derivatives. This result confirms the versatility
and broader applicability of our approach.

The structural blueprint of this study is mapped as follows. Some
geometric preliminaries is introduced in Section \ref{ozppappa}.
The hyperbolic system is formulated in Section \ref{sec:Mixed-problem-and}.
Section \ref{sec:Integral-Operators-and} reformulates the hyperbolic
system into an equivalent fixed-point problem via integral operators.
Section \ref{sec:Local-solution} discusses the uniqueness and existence
of solutions.

\section{Preliminaries}\label{ozppappa}

Let $K$ be a compact convex non-empty subset of $\R^{d},$$d$ integer,
$d\geq1$ and let $\mathcal{H}\equiv\{H_{m}\}^{m_{0}}_{m=1}$ be a
finite family of pair-wise distinct affine hyper-planes of $\R^{d}$.
It is a standard result in the theory of hyper-plane arrangements
that the connected components of $\Omega=\R^{d}\setminus\bigcup^{m_{0}}_{m=1}H_{i}$
form a finite collection of open convex sets. Thus, the connected
components of $\ensuremath{\operatorname{Int}(K)\setminus\bigcup^{m_{0}}_{i=1}H_{i}}$
are exactly given by the non-empty intersections of $\ensuremath{\operatorname{Int}(K)}$
with the connected components of $\Omega$. Since $\operatorname{Int}(K)$
is open and convex, it follows that $\ensuremath{\operatorname{Int}(K)\setminus\bigcup^{m_{0}}_{i=1}H_{i}}$
also consists of a finite collection of non-empty, open, and convex
connected components. 
\begin{lem}
\noindent\label{lemosp}Let $v\in\mathbb{R}^{d}$ be a non-zero vector
that is not in the direction space of any hyper-plane $H_{m}$ for
all $i\in\{1,\dots,m_{0}\}$. For any segment $[x,y]\subset\mathbb{R}^{d}$
and for a sufficiently small $t\in\mathbb{R}\setminus\{0\}$, the
translate of segment $[x,y]$ by the vector $tv$ cannot be contained
in any hyper-plane $H_{m}$. 
\end{lem}

\begin{proof}
Let us denote the canonical dot product in $\mathbb{R}^{d}$ by $\left\langle \cdot,\cdot\right\rangle $.
The affine hyper-plane $H_{m}$ in $\mathbb{R}^{d}$ can be defined
by 
\begin{equation}
H_{m}=\{z\in\mathbb{R}^{d}\mid\langle a_{m},z\rangle=b_{m}\}\label{oolpm}
\end{equation}
where $a_{m}\in\mathbb{R}^{d}\setminus\{0\}$ and $b_{m}\in\mathbb{R}$
. The direction space of $H_{m}$, denoted by $\vec{H}_{m}$, is given
by 
\begin{equation}
\vec{H}_{m}=\{w\in\mathbb{R}^{d}\mid\langle a_{m},w\rangle=0\}.\label{rfole}
\end{equation}
By hypothesis, for every $m\in\{1,\dots,N\}$, we have $\langle a_{m},v\rangle\neq0.$
The translate of segment $[x,y]$ by the vector $tv$ is given by
\begin{equation}
[x,y]+tv=\{(1-\lambda)x+\lambda y+tv\mid\lambda\in[0,1]\}\label{oluju}
\end{equation}
The endpoint of $[x,y]+tv$ given by $\lambda=0$ is $x+tv.$ That
endpoint lies in $H_{m}$ \eqref{oolpm} iff 
\begin{align}
 & \langle a_{m},x\rangle+t\langle a_{m},v\rangle=b_{m}\label{eq:left_end}
\end{align}
Since $\langle a_{m},v\rangle\neq0,$ \eqref{eq:left_end}$\iff$
$t=t_{m}\equiv\frac{b_{m}-\langle a_{m},x\rangle}{\langle a_{m},v\rangle}.$
\\
 Let $I$ be the subset of indices where $t_{m}\neq0$. \\
 If $I$ is empty, then $t_{m}=0$ for all $m\in\{1,\dots,m_{0}\}$.
In this case, for any $t\neq0,$the endpoint $x+tv$ does not lie
in $H_{m},$for all $m\in\{1,\dots,m_{0}\}$. \\
 Suppose $I$ is non-empty; let us consider the set of absolute values
$\{|t_{m}|:m\in I\}$. Let $t_{\min}\equiv{\displaystyle \min_{m\in I}}|t_{m}|>0$.
Now, choose any non-zero real number $t$ such that $|t|<t_{\min}$.
\\
 For indices $m$ where $t_{m}=0$, our choice of $t\neq0$ ensures
$t\neq t_{m}$. For indices $m$ in $I$, our choice ensures $|t|<t_{\min}\le|t_{m}|$,
which implies $t\neq t_{m}$.\\
 Since $t\neq t_{m}$ for all $m\in\{1,\dots,m_{0}\}$, the endpoint
$x+tv$ does not belong to any $H_{m}$. \\
 In conclusion, for any $0<|t|<t_{\min},$ the translate $[x,y]+tv$
of segment $[x,y]$ by the vector $tv$ cannot be contained within
any hyperplane $H_{m}$ . 
\end{proof}

\begin{lem}
\label{lempmm}Let $v$ a non-zero vector in $\mathbb{R}^{d}.$ Let
$x,y$ be two points in the interior of K. If the length of v is sufficiently
small, then the translate of segment $[x,y]$ by the vector $v$ is
in the interior of K. 
\end{lem}

\begin{proof}
Since $\operatorname{int}(K)$ is convex, we have $[x,y]\subset\operatorname{int}(K).$
For each $z\in[x,y]$, there exists $r_{z}>0$ such that $B(z,r_{z})\subset\operatorname{int}(K).$
Hence the family $\left\{ B\left(z,\frac{r_{z}}{2}\right):z\in[x,y]\right\} $
is an open cover of the compact segment $[x,y]$. Therefore, there
exist finitely many points $z_{1},\ldots,z_{P}\in[x,y]$ such that
$[x,y]\subset\bigcup^{P}_{i=1}B\left(z_{i},\frac{r_{i}}{2}\right),$
where $r_{i}=r_{z_{i}}$. Set $\delta=\frac{1}{2}{\displaystyle \min_{1\le i\le P}r_{i}}>0.$
Let $v$ satisfy $\|v\|<\delta$, and let $w\in[x,y]$. Then there
exists $i\in\{1,\ldots,P\}$ such that $\|w-z_{i}\|<\frac{r_{i}}{2}.$
We have 
\[
\|(w+v)-z_{i}\|\le\|w-z_{i}\|+\|v\|<\frac{r_{i}}{2}+\delta\le r_{i}.
\]
Hence, $w+v\in B(z_{i},r_{i})\subset\operatorname{int}(K).$ It follows
that $[x,y]+v\subset\operatorname{int}(K).$ 
\end{proof}

\begin{notation}
\label{wqjjqu}For any bounded real-valued function $f$ with domain
$D_{f}$, we set $\|f\|_{\infty}\equiv{\displaystyle \sup_{x\in D_{f}}}|f(x)|$.
Furthermore, if $D_{f}\subset\mathbb{R}^{d}$ and $f\equiv f(x_{1},\dots,x_{d})$
is bounded along with its first-order partial derivatives ${\displaystyle \frac{\partial f}{\partial x_{\alpha}}}$
(for $\alpha=1,\dots,d$), we define:\\
 $\|f\|_{1}\equiv\max\left\{ \|f\|_{\infty},\,{\displaystyle \max_{1\le\alpha\le d}}\left\Vert \frac{\partial f}{\partial x_{\alpha}}\right\Vert _{\infty}\right\} .$
\end{notation}

\begin{notation}
\label{iqiqjja}Given an integer $p\ge2$ and a vector-valued function
$F=(f_{1},\dots,f_{p})$ whose components are real-valued and bounded,
we equip it with the norm: 
\begin{equation}
\|F\|\doteq{\displaystyle \max_{1\le i\le p}}\|f_{i}\|_{\infty}.\label{oosaz}
\end{equation}
\end{notation}

\begin{notation}
\label{kqqkkd} Let $C\left(K,\R\right)$, the space of continuous
functions from $K$ on $\R,$ be endowed with $\left\Vert \cdot\right\Vert _{\infty}$.
Let $\mathscr{E}_{\mathcal{H}}$ be the subspace of $C\left(K,\R\right)$
consisting of functions $f$ such that $\forall\alpha=1,\cdots,d,$$\dfrac{\partial f}{\partial x_{\alpha}}$
is defined everywhere on $\operatorname{int}(K)\setminus{\displaystyle \cup^{m_{0}}_{m=1}}H_{m}$
and is continuous and bounded.
\end{notation}

\begin{notation}
Let $p\geq1$ be an integer and $C\left(K,\R\right)^{p}$ be equipped
with the norm $\left\Vert \cdot\right\Vert $(\ref{oosaz}). For $F=\left(f_{i}\right)^{p}_{i=1}\in\left(\mathscr{E}_{\mathcal{H}}\right)^{p},$
let 
\begin{align}
\mathscr{N}\left(F\right) & \equiv\max\left\{ \right.\left\Vert F\right\Vert ,{\displaystyle \max_{1\le\alpha\le d}}\left\Vert \dfrac{\partial F}{\partial x_{\alpha}}\right\Vert \left.\right\} \label{bxhdg}\\
 & ={\displaystyle \max_{1\le i\le p}}\|f_{i}\|_{1}.\nonumber 
\end{align}
\end{notation}

\begin{defn}
\label{nnjnjn} For $R>0,$ let us define $\mathscr{M}_{R}\equiv\left\{ \right.F\in\left(\mathscr{E}_{\mathcal{H}}\right)^{p}/\mathscr{N}\left(F\right)\leq R\left.\right\} .$ 
\end{defn}

\begin{prop}
\label{prop::odlp-1} For every $R>0,$ $\mathscr{M}_{R}$ is a non-empty
convex subset of \textup{$C\left(K,\R\right)^{p}$ }and is relatively
compact in $\left(C\left(K,\R\right)^{p},\left\Vert \cdot\right\Vert \right).$ 
\end{prop}

\begin{proof}
The null function is in $\mathscr{M}_{R}.$ The convexity of $\mathscr{M}_{R}$
is immediate from its definition. \\
 1) Let $x\in K$ be fixed. $\forall F\in\mathscr{M}_{R}$ we have
$\left\Vert F\left(x\right)\right\Vert _{\R^{p}}\leq\left\Vert F\right\Vert \leq\mathscr{N}\left(F\right)\leq R$
(with $\left\Vert \left(y_{j}\right)^{p}_{j=1}\right\Vert _{\R^{p}}\equiv{\displaystyle \max_{j}}\left|y_{j}\right|$).
Thus $\left\{ F\left(x\right)/F\in\mathscr{M}_{R}\right\} $ is bounded
in $\R^{p},$ hence it is relatively compact in $\R^{p}.$ \\
 2) Let us prove that $\mathscr{M}_{R}$ is equicontinuous.\\
 We have $\forall F=\left(f_{i}\right)^{p}_{i=1}\in\mathscr{M}_{R},$$\forall i=1,\cdots,p,\forall\alpha=1,\cdots,d,$
$\left\Vert \dfrac{\partial f_{i}}{\partial x_{\alpha}}\right\Vert _{\infty}\leq\left\Vert \dfrac{\partial F}{\partial x_{\alpha}}\right\Vert \leq\mathscr{N}\left(F\right)\leq R.$
Thus for all $x\in\operatorname{int}(K)\setminus{\displaystyle \cup^{m_{0}}_{m=1}}H_{m},$
we have $\forall i,\left\Vert \dfrac{\partial f_{i}}{\partial x_{\alpha}}\left(x\right)\right\Vert _{\infty}\leq R.$\\
 Let $F\in\mathscr{M}_{R}$ be fixed. \\
 Let $x,y\in\operatorname{int}(K).$ Assume that the segment $[x,y]$
is not contained in any hyperplane of $\mathcal{H}$. A line intersects
an affine hyperplane in at most one point unless it is contained in
that hyperplane. Thus $[x,y]\cap\left({\displaystyle \cup^{m_{0}}_{m=1}}H_{m}\right)$
is finite. Hence there exist points $x=z_{0},z_{1},\ldots,z_{r}=y,$
ordered along the segment, such that each open subsegment $\left]z_{k},z_{k+1}\right[$
is contained in one connected component of $\ensuremath{\mathring{K}\setminus{\displaystyle \cup^{m_{0}}_{m=1}}H_{m}}.$
It follows from the mean value theorem that there exist a constant
$\mu>0$ such that $\left\Vert F(z_{k+1})-F(z_{k})\right\Vert \le\mu R\left\Vert z_{k+1}-z_{k}\right\Vert _{\mathbb{R}^{d}},\qquad k=0,\ldots,r-1.$
By continuity of $F$, the above inequality also holds when one endpoint
belongs to any of the hyperplanes. Summing over $k$,\\
 $\left\Vert F(x)-F(y)\right\Vert \le\mu R\sum^{r-1}_{k=0}\left\Vert z_{k+1}-z_{k}\right\Vert =\mu R\left\Vert x-y\right\Vert $
as $z_{k}\in\left[z_{k-1},z_{k+1}\right],$$k=1,\cdots,r-1.$ \\
 Assume now that $[x,y]\subset H_{l}$ for some $H_{l}\in\mathcal{H}$.
Let $v\in\mathbb{R}^{d}$ be a non-zero vector that is not in the
direction space of any hyper-plane $H_{m}.$ From lemmas \eqref{lemosp}
and \eqref{lempmm}, there exist $t_{min},\delta>0$ such that $\forall t\in\mathbb{R}\setminus\left\{ 0\right\} ,$
if $\left|t\right|<t_{min}$ and $\left\Vert tv\right\Vert <\delta,$
then the translate of segment $[x,y]$ by the vector $tv$ is in the
interior of $K$ and cannot be contained in any of hyper-planes $H_{m}$.
Take $t\in\mathbb{R}\setminus\left\{ 0\right\} $ such $\left|t\right|<t_{min}$
and $\left\Vert tv\right\Vert <\delta.$ Then for all $n\in\mathbb{N},n\geq1,$we
have $\left|{\displaystyle \frac{t}{n}}\right|<t_{min}$ and $\left\Vert {\displaystyle \frac{t}{n}}v\right\Vert <\delta.$Thus
the translate of segment $[x,y]$ by the vector ${\displaystyle \frac{t}{n}}v$
is in the interior of $K$ and cannot be contained in any of hyper-planes
$H_{m}$. \\
 Applying the previous case, we have for all $n\geq1,$$\left\Vert F(x+{\displaystyle \frac{t}{n}}v)-F(y+{\displaystyle \frac{t}{n}}v)\right\Vert \le\mu R\left\Vert x-y\right\Vert .$\\
 Passing to the limit as $n\to\infty$ and using the continuity of
$F$, we have $\left\Vert F(x)-F(y)\right\Vert \le\mu R\left\Vert x-y\right\Vert .$\\
 Therefore, $\left\Vert F(x)-F(y)\right\Vert \le\mu R\left\Vert x-y\right\Vert $
for all $x,y\in\operatorname{int}(K).$\\
 Now, let $x,y\in\partial K$. As $K$ is convex there exist sequences
$x_{n},y_{n}\in\operatorname{int}(K)$ satisfying $x_{n}\to x,y_{n}\to y.$
Applying the previous estimate, we have for all $n\geq1,$ $\left\Vert F(x_{n})-F(y_{n})\right\Vert \le\mu R\left\Vert x_{n}-y_{n}\right\Vert .$
Passing to the limit yields $\left\Vert F(x)-F(y)\right\Vert \le\mu R\left\Vert x-y\right\Vert .$
\\
 Finally, we conclude that for all $x\in K,$ 
\begin{align}
\forall\varepsilon & >0,\exists\eta>0,\forall F\in\mathscr{M}_{R},\forall y\in K:\left\Vert x-y\right\Vert <\eta\implies\left\Vert F\left(x\right)-F\left(y\right)\right\Vert \leq\varepsilon.
\end{align}
Thus $\mathscr{M}_{R}$ is equicontinuous in $C\left(K,\R\right)^{p}.$

By the Arzelà-Ascoli's theorem, $\mathscr{M}_{R}$ is relatively compact
in $\left(C\left(K,\R\right)^{p},\left\Vert \cdot\right\Vert \right)$
. 
\end{proof}

\begin{cor}
\label{prop::odlp-1-1} For every $R>0$, the subset of non-negative
elements of $\mathscr{M}_{R}$, defined by $\mathscr{M}^{+}_{R}\equiv\left\{ F\in\mathscr{M}_{R}/F\geq0\right\} $
is a non-empty, convex, and relatively compact subset of the normed
space $\left(C(K,\mathbb{R})^{p},\left\Vert \cdot\right\Vert \right)$. 
\end{cor}

\begin{proof}
This follows immediately from the fact that $\mathscr{M}^{+}_{R}$
is the intersection of two convex subsets containing the zero function,
and is contained within the relatively compact subset $\mathscr{M}_{R}$
of $\left(C(K,\mathbb{R})^{p},\left\Vert \cdot\right\Vert \right)$. 
\end{proof}

In the sequel, $p=4$ and $d=3.$

\section{Mixed problem and main results}\label{sec:Mixed-problem-and}

\subsection{Mixed problem}

In this work, we consider the general four-velocity Broadwell model,
whose dynamics are governed by the following kinetic equations \eqref{b2ksoauau}:
\begin{equation}
\begin{cases}
{\displaystyle \frac{\partial N_{1}}{\partial t}+c\cos\theta\frac{\partial N_{1}}{\partial x}+c\sin\theta\frac{\partial N_{1}}{\partial y}=Q\left(N\right)}\\
{\displaystyle \frac{\partial N_{2}}{\partial t}-c\sin\theta\frac{\partial N_{2}}{\partial x}+c\cos\theta\frac{\partial N_{2}}{\partial y}=-Q\left(N\right)}\\
{\displaystyle \frac{\partial N_{3}}{\partial t}+c\sin\theta\frac{\partial N_{3}}{\partial x}-c\cos\theta\frac{\partial N_{3}}{\partial y}=-Q\left(N\right)}\\
{\displaystyle \frac{\partial N_{4}}{\partial t}-c\cos\theta\frac{\partial N_{4}}{\partial x}-c\sin\theta\frac{\partial N_{4}}{\partial y}=Q\left(N\right)}
\end{cases}\label{b2ksoauau}
\end{equation}

\begin{equation}
Q\left(N\right)=2cS\left(N_{2}N_{3}-N_{1}N_{4}\right).\label{eq:solls}
\end{equation}
for $\theta\in\left]0;\dfrac{\pi}{2}\right[$ where the unknowns $N_{1},N_{2},N_{3}$
and $N_{4}$ are real-valued functions of $(t,x,y)\in\R^{3}.$ 
\begin{defn}
In the sequel, $\mathcal{P}$ denotes the initial-boundary value problem
consisting of system \eqref{b2ksoauau} posed in $\left[0;+\infty\right[\times\left[a_{1},b_{1}\right]\times\left[a_{2},b_{2}\right]\subset\mathbb{R}^{3}$,
subject to the initial and boundary conditions \eqref{eq:lsos-2}-\eqref{eq:oiiap-1}: 
\end{defn}

\selectlanguage{french}%
\begin{align}
N_{i}\left(0,x,y\right)= & N^{0}_{i}\left(x,y\right),\;\left(x,y\right)\in\left[a_{1};b_{1}\right]\times\left[a_{2};b_{2}\right],i=1,\cdots,4\label{eq:lsos-2}\\
N_{1}\left(t,a_{1},y\right)= & N^{-}_{1}\left(t,y\right),\;\left(t,y\right)\in\left[0;+\infty\right[\times\left[a_{2};b_{2}\right]\label{eq:losso-2}\\
N_{1}\left(t,x,a_{2}\right)= & N^{--}_{1}\left(t,x\right),\;\left(t,x\right)\in\left[0;+\infty\right[\times\left[a_{1};b_{1}\right]\label{eq:mppos-1}\\
N_{2}\left(t,b_{1},y\right)= & N^{+}_{2}\left(t,y\right),\;\left(t,y\right)\in\left[0;+\infty\right[\times\left[a_{2};b_{2}\right]\label{eq:mpsss-1}\\
N_{2}\left(t,x,a_{2}\right)= & N^{--}_{2}\left(t,x\right),\;\left(t,x\right)\in\left[0;+\infty\right[\times\left[a_{1};b_{1}\right]\label{eq:lsoo-1-1}\\
N_{3}\left(t,a_{1},y\right)= & N^{-}_{3}\left(t,y\right),\;\left(t,y\right)\in\left[0;+\infty\right[\times\left[a_{2};b_{2}\right]\label{ikis-1}\\
N_{3}\left(t,x,b_{2}\right)= & N^{++}_{3}\left(t,x\right),\;\left(t,x\right)\in\left[0;+\infty\right[\times\left[a_{1};b_{1}\right]\label{eq:ikioi-3}\\
N_{4}\left(t,b_{1},y\right)= & N^{+}_{4}\left(t,y\right),\;\left(t,y\right)\in\left[0;+\infty\right[\times\left[a_{2};b_{2}\right]\label{eq:ikioi-1-2}\\
N_{4}\left(t,x,b_{2}\right)= & N^{++}_{4}\left(t,x\right),\;\left(t,x\right)\in\left[0;+\infty\right[\times\left[a_{1};b_{1}\right]\label{eq:oiiap-1}
\end{align}

\selectlanguage{english}%
The prescribed initial and boundary data are assumed to enjoy sufficient
regularity and consistency properties: they are non negative, continuously
differentiable with bounded first-order partial derivatives, and satisfy
the compatibility conditions \eqref{eq:kqiiq-3-1}-\eqref{eq:mqppa-1}\foreignlanguage{french}{
\begin{align}
N^{0}_{1}\left(a_{1},y\right) & =N^{-}_{1}\left(0,y\right),\;y\in\left[a_{2};b_{2}\right]\label{eq:kqiiq-3-1}\\
N^{0}_{1}\left(x,a_{2}\right) & =N^{--}_{1}\left(0,x\right),\;x\in\left[a_{1};b_{1}\right]\label{eq:ksos-1}\\
N^{-}_{1}\left(t,a_{2}\right) & =N^{--}_{1}\left(t,a_{1}\right),\;t\in\left[0;+\infty\right[\label{eq:slos-1-1-1}
\end{align}
\begin{align}
N^{0}_{2}\left(b_{1},y\right) & =N^{+}_{2}\left(0,y\right),\;y\in\left[a_{2};b_{2}\right]\label{eq:qmpp-1}\\
N^{0}_{2}\left(x,a_{2}\right) & =N^{--}_{2}\left(0,x\right),\;x\in\left[a_{1};b_{1}\right]\label{eq:lqoop-2}\\
N^{+}_{2}\left(t,a_{2}\right) & =N^{--}_{2}\left(t,b_{1}\right),\;t\in\left[0;+\infty\right[\label{eq:mqpp-2}
\end{align}
\begin{align}
N^{0}_{3}\left(a_{1},y\right) & =N^{-}_{3}\left(0,y\right),\;y\in\left[a_{2};b_{2}\right]\label{eq:lopq-1}\\
N^{0}_{3}\left(x,b_{2}\right) & =N^{++}_{3}\left(0,x\right),\;x\in\left[a_{1};b_{1}\right]\label{eq:loppq-1-2}\\
N^{-}_{3}\left(t,b_{2}\right) & =N^{++}_{3}\left(t,a_{1}\right),\;t\in\left[0;+\infty\right[\label{eq:looqp-2}
\end{align}
}

\selectlanguage{french}%
\begin{align}
N^{0}_{4}\left(b_{1},y\right) & =N^{+}_{4}\left(0,y\right),\;y\in\left[a_{2};b_{2}\right]\label{eq:looqp-1-2}\\
N^{0}_{4}\left(x,b_{2}\right) & =N^{++}_{4}\left(0,x\right),\;x\in\left[a_{1};b_{1}\right]\label{eq:mqpp-1-1}\\
N^{+}_{4}\left(t,b_{2}\right) & =N^{++}_{4}\left(t,b_{1}\right),\;t\in\left[0;+\infty\right[\label{eq:mqppa-1}
\end{align}

\selectlanguage{english}%

\begin{defn}
Now given $\left[\tau;\tau'\right]\subset\R_{+}$, let $K_{\tau,\tau'}\equiv\left[\tau;\tau'\right]\times\left[a_{1};b_{1}\right]\times\left[a_{2};b_{2}\right].$
The problem $\mathcal{P}_{\tau,\tau'}$ is defined as the system \eqref{eq:erter-1}-\eqref{eq:zzzfra-1}
subject to the initial and boundary conditions \eqref{eq:lsos}-\eqref{eq:oiiap}: 
\end{defn}

\begin{align}
{\textstyle \dfrac{\p N_{1}}{\p t}+c\cos\theta\dfrac{\p N_{1}}{\p x}} & {\displaystyle +c\sin\theta\dfrac{\p N_{1}}{\p y}}=Q(N),\;\left(t,x,y\right)\in\mathring{K_{\tau,\tau'}}\label{eq:erter-1}\\
\dfrac{\p N_{2}}{\p t}-c\sin\theta\dfrac{\p N_{2}}{\p x} & +c\cos\theta\dfrac{\p N_{2}}{\p y}=-Q(N),\;\left(t,x,y\right)\in\mathring{K_{\tau,\tau'}}\\
\dfrac{\p N_{3}}{\p t}+c\sin\theta\dfrac{\p N_{3}}{\p x} & -c\cos\theta\dfrac{\p N_{3}}{\p y}=-Q(N),\;\left(t,x,y\right)\in\mathring{K_{\tau,\tau'}}\\
\dfrac{\p N_{4}}{\p t}-c\cos\theta\dfrac{\p N_{4}}{\p x} & -c\sin\theta\dfrac{\p N_{4}}{\p y}=Q(N),\;\left(t,x,y\right)\in\mathring{K_{\tau,\tau'}}\label{eq:zzzfra-1}\\
N_{i}\left(\tau,x,y\right)= & N^{\tau}_{i}\left(x,y\right),\;\left(x,y\right)\in\left[a_{1};b_{1}\right]\times\left[a_{2};b_{2}\right],i=1,\cdots,4\label{eq:lsos}\\
N_{1}\left(t,a_{1},y\right)= & N^{-}_{1}\left(t,y\right),\;\left(t,y\right)\in\left[\tau;\tau'\right]\times\left[a_{2};b_{2}\right]\label{eq:losso}\\
N_{1}\left(t,x,a_{2}\right)= & N^{--}_{1}\left(t,x\right),\;\left(t,x\right)\in\left[\tau;\tau'\right]\times\left[a_{1};b_{1}\right]\label{eq:mppos}\\
N_{2}\left(t,b_{1},y\right)= & N^{+}_{2}\left(t,y\right),\;\left(t,y\right)\in\left[\tau;\tau'\right]\times\left[a_{2};b_{2}\right]\label{eq:mpsss}\\
N_{2}\left(t,x,a_{2}\right)= & N^{--}_{2}\left(t,x\right),\;\left(t,x\right)\in\left[\tau;\tau'\right]\times\left[a_{1};b_{1}\right]\label{eq:lsoo-1}\\
N_{3}\left(t,a_{1},y\right)= & N^{-}_{3}\left(t,y\right),\;\left(t,y\right)\in\left[\tau;\tau'\right]\times\left[a_{2};b_{2}\right]\label{ikis}\\
N_{3}\left(t,x,b_{2}\right)= & N^{++}_{3}\left(t,x\right),\;\left(t,x\right)\in\left[\tau;\tau'\right]\times\left[a_{1};b_{1}\right]\label{eq:ikioi}\\
N_{4}\left(t,b_{1},y\right)= & N^{+}_{4}\left(t,y\right),\;\left(t,y\right)\in\left[\tau;\tau'\right]\times\left[a_{2};b_{2}\right]\label{eq:ikioi-1}\\
N_{4}\left(t,x,b_{2}\right)= & N^{++}_{4}\left(t,x\right),\;\left(t,x\right)\in\left[\tau;\tau'\right]\times\left[a_{1};b_{1}\right]\label{eq:oiiap}
\end{align}

The data $N^{\tau}_{i},\left(i=1,2,3,4\right),$ $N^{-}_{1},N^{--}_{1},$$N^{+}_{2},N^{--}_{2},$$N^{-}_{3},N^{++}_{3},$$N^{+}_{4},N^{++}_{4}$
are assumed to be non-negative, $C^{1}$-regular with bounded derivatives,
and subject to the compatibility conditions \eqref{eq:kqiiq-3}-\eqref{eq:mqppa}:
\begin{align}
N^{\tau}_{1}\left(a_{1},y\right) & =N^{-}_{1}\left(\tau,y\right),\;y\in\left[a_{2};b_{2}\right]\label{eq:kqiiq-3}\\
N^{\tau}_{1}\left(x,a_{2}\right) & =N^{--}_{1}\left(\tau,x\right),\;x\in\left[a_{1};b_{1}\right]\label{eq:ksos}\\
N^{-}_{1}\left(t,a_{2}\right) & =N^{--}_{1}\left(t,a_{1}\right),\;t\in\left[\tau;\tau'\right]\label{eq:slos-1-1}
\end{align}
\begin{align}
N^{\tau}_{2}\left(b_{1},y\right) & =N^{+}_{2}\left(\tau,y\right),\;y\in\left[a_{2};b_{2}\right]\label{eq:qmpp}\\
N^{\tau}_{2}\left(x,a_{2}\right) & =N^{--}_{2}\left(\tau,x\right),\;x\in\left[a_{1};b_{1}\right]\label{eq:lqoop}\\
N^{+}_{2}\left(t,a_{2}\right) & =N^{--}_{2}\left(t,b_{1}\right),\;t\in\left[\tau;\tau'\right]\label{eq:mqpp}
\end{align}
\begin{align}
N^{\tau}_{3}\left(a_{1},y\right) & =N^{-}_{3}\left(\tau,y\right),\;y\in\left[a_{2};b_{2}\right]\label{eq:lopq}\\
N^{\tau}_{3}\left(x,b_{2}\right) & =N^{++}_{3}\left(\tau,x\right),\;x\in\left[a_{1};b_{1}\right]\label{eq:loppq-1}\\
N^{-}_{3}\left(t,b_{2}\right) & =N^{++}_{3}\left(t,a_{1}\right),\;t\in\left[\tau;\tau'\right]\label{eq:looqp}
\end{align}

\begin{align}
N^{\tau}_{4}\left(b_{1},y\right) & =N^{+}_{4}\left(\tau,y\right),\;y\in\left[a_{2};b_{2}\right]\label{eq:looqp-1}\\
N^{\tau}_{4}\left(x,b_{2}\right) & =N^{++}_{4}\left(\tau,x\right),\;x\in\left[a_{1};b_{1}\right]\label{eq:mqpp-1}\\
N^{+}_{4}\left(t,b_{2}\right) & =N^{++}_{4}\left(t,b_{1}\right),\;t\in\left[\tau;\tau'\right]\label{eq:mqppa}
\end{align}

\begin{notation}
\label{jskskk}We introduce the following functions 
\begin{align*}
\overline{N^{\tau}_{1}}\left(t,x,y\right) & \equiv N^{\tau}_{1}\left(x-c\left(t-\tau\right)\cos\theta,y-c\left(t-\tau\right)\sin\theta\right)\\
\overline{N^{-}_{1}}\left(t,x,y\right) & \equiv N^{-}_{1}\left(t-\frac{1}{c\cos\theta}x+\frac{a_{1}}{c\cos\theta},-\frac{\sin\theta}{\cos\theta}x+y+\frac{\sin\theta}{\cos\theta}a_{1}\right)\\
\overline{N^{--}_{1}}\left(t,x,y\right) & \equiv N^{--}_{1}\left(t-\frac{1}{c\sin\theta}y+\frac{a_{2}}{c\sin\theta},x-\frac{\cos\theta}{\sin\theta}y+\frac{\cos\theta}{\sin\theta}a_{2}\right)
\end{align*}
\begin{align*}
\overline{N^{\tau}_{2}}\left(t,x,y\right) & \equiv N^{\tau}_{2}\left(x+c\left(t-\tau\right)\sin\theta,y-c\left(t-\tau\right)\cos\theta\right)\\
\overline{N^{+}_{2}}\left(t,x,y\right) & \equiv N^{+}_{2}\left(t+\frac{1}{c\sin\theta}x-\frac{b_{1}}{c\sin\theta},\frac{\cos\theta}{\sin\theta}x+y-\frac{\cos\theta}{\sin\theta}b_{1}\right)\\
\overline{N^{--}_{2}}\left(t,x,y\right) & \equiv N^{--}_{2}\left(t-\frac{1}{c\cos\theta}y+\frac{a_{2}}{c\cos\theta},x+\frac{\sin\theta}{\cos\theta}y-\frac{\sin\theta}{\cos\theta}a_{2}\right)
\end{align*}
\begin{align*}
\overline{N^{\tau}_{3}}\left(t,x,y\right) & \equiv N^{\tau}_{3}\left(x-c\left(t-\tau\right)\sin\theta,y+c\left(t-\tau\right)\cos\theta\right)\\
\overline{N^{-}_{3}}\left(t,x,y\right) & \equiv N^{-}_{3}\left(t-\frac{1}{c\sin\theta}x+\frac{a_{1}}{c\sin\theta},\frac{\cos\theta}{\sin\theta}x+y-\frac{\cos\theta}{\sin\theta}a_{1}\right)\\
\overline{N^{++}_{3}}\left(t,x,y\right) & \equiv N^{++}_{3}\left(t+\frac{1}{c\cos\theta}y-\frac{b_{2}}{c\cos\theta},x+\frac{\sin\theta}{\cos\theta}y-\frac{\sin\theta}{\cos\theta}b_{2}\right)
\end{align*}
\begin{align*}
\overline{N^{\tau}_{4}}\left(t,x,y\right) & \equiv N^{\tau}_{4}\left(x+c\left(t-\tau\right)\cos\theta,y+c\left(t-\tau\right)\sin\theta\right)\\
\overline{N^{+}_{4}}\left(t,x,y\right) & \equiv N^{+}_{4}\left(t+\frac{1}{c\cos\theta}x-\frac{b_{1}}{c\cos\theta},-\frac{\sin\theta}{\cos\theta}x+y+\frac{\sin\theta}{\cos\theta}b_{1}\right)\\
\overline{N^{++}_{4}}\left(t,x,y\right) & \equiv N^{++}_{4}\left(t+\frac{1}{c\sin\theta}y-\frac{b_{2}}{c\sin\theta},x-\frac{\cos\theta}{\sin\theta}y+\frac{\cos\theta}{\sin\theta}b_{2}\right)
\end{align*}
\end{notation}

\begin{notation}
\label{gjjgkk} Let us introduce the constant $\mathfrak{q}$ which
is associated to the problem$\mathcal{P}_{\tau,\tau'}:$ 
\begin{multline}
\mathfrak{q}\equiv\max_{1\leq i\leq4}\Biggl\{\left\Vert \overline{N^{\tau}_{i}}\right\Vert _{1},\left\Vert \overline{N^{-}_{1}}\right\Vert _{1},\left\Vert \overline{N^{--}_{1}}\right\Vert _{1},\left\Vert \overline{N^{+}_{2}}\right\Vert _{1},\\
\left\Vert \overline{N^{--}_{2}}\right\Vert _{1},\left\Vert \overline{N^{-}_{3}}\right\Vert _{1},\left\Vert \overline{N^{++}_{3}}\right\Vert _{1},\left\Vert \overline{N^{+}_{4}}\right\Vert _{1},\left\Vert \overline{N^{++}_{4}}\right\Vert _{1}\Biggr\}.\label{eq:loso-1-1-2}
\end{multline}
\end{notation}

\begin{notation}
\label{rrrrt}Let us introduce the constants $\mu$, $\lambda$, and
$\delta$ associated with the discrete kinetic model: 
\begin{align}
\mu & \equiv8+\frac{4}{c\cos\theta}+\frac{4}{c\sin\theta}+8c\cos\theta+8c\sin\theta,\\
\lambda & \equiv32+\frac{16}{c\cos\theta}+\frac{16}{c\sin\theta}+16c\cos\theta+16c\sin\theta,\\
\delta & \equiv8+\frac{4}{c\cos\theta}+\frac{4}{c\sin\theta}+4c\cos\theta+4c\sin\theta.
\end{align}
Let $R_{0}>0$ be fixed. Furthermore, we define $\mathfrak{p}_{\sigma}$
and $\mathfrak{q}_{\sigma}$ as follows: 
\begin{equation}
\mathfrak{p}_{\sigma}\equiv\left(\mu+\lambda\sigma R_{0}(\tau'-\tau)\right)(\sigma+cS)\label{ooodl}
\end{equation}
and 
\begin{equation}
\mathfrak{q}_{\sigma}\equiv\mathfrak{q}(1+\delta\sigma R_{0}).\label{ppmmzp}
\end{equation}
\end{notation}

\begin{notation}
Let us introduce the parameter 
\begin{equation}
\gamma\equiv1+c\cos\theta+c\sin\theta+{\displaystyle \frac{1}{c\cos\theta}+\frac{\sin\theta}{\cos\theta}+\frac{1}{c\sin\theta}+\frac{\cos\theta}{\sin\theta}}.\label{xwccv}
\end{equation}
Put 
\begin{equation}
f_{\sigma}\left(R_{0}\right)\equiv\dfrac{1}{4\mu\left(1+\delta\sigma R_{0}\right)\left(\sigma+cS\right)}\label{eq:loi-1}
\end{equation}
and 
\begin{align}
g\left(\mathfrak{q}\right) & \equiv\min\left\{ 1;\dfrac{1}{\lambda\sigma R_{0}}\left(\dfrac{1}{4\mathfrak{q}\left(1+\delta\sigma R_{0}\right)}-\mu\right)\right\} \nonumber \\
 & =\min\left\{ 1;\dfrac{\mu}{\lambda\sigma R_{0}}\left(\right.\dfrac{f_{\sigma}\left(R_{0}\right)}{\mathfrak{q}}-1\left.\right)\right\} \label{kqiqiiq}
\end{align}
\end{notation}

\begin{notation}
We introduce the function 
\begin{align}
\varPsi\left(\sigma\right) & \equiv\dfrac{{\textstyle -1+\sqrt{1+\frac{\delta\sigma}{\mu\left(\sigma+cS\right)\gamma}}}}{2\delta\sigma}\label{sskks}\\
 & =\dfrac{{\textstyle -1+\sqrt{1+\frac{\left(8+\frac{4}{c\cos\theta}+\frac{4}{c\sin\theta}+4c\cos\theta+4c\sin\theta\right)\sigma}{\left(8+\frac{4}{c\cos\theta}+\frac{4}{c\sin\theta}+8c\cos\theta+8c\sin\theta\right)\left(\sigma+cS\right)\left(1+c\cos\theta+c\sin\theta+\frac{1}{c\cos\theta}+\frac{\sin\theta}{\cos\theta}+\frac{1}{c\sin\theta}+\frac{\cos\theta}{\sin\theta}\right)}}}}{2\left(8+\frac{4}{c\cos\theta}+\frac{4}{c\sin\theta}+4c\cos\theta+4c\sin\theta\right)\sigma}
\end{align}
\end{notation}

In the section~\ref{ozppappa}, let us consider $K=K_{\tau,\tau'}\subset\mathbb{R}^{3}$
with coordinates $x_{1}\equiv t$, $x_{2}\equiv x$, and $x_{3}\equiv y$.
Let $\mathcal{H}=\{H_{m}\}^{12}_{m=1}$ be the family of planes defined
by: {\small
\begin{align}
H_{1} & :x-c(t-\tau)\cos\theta=a_{1}, & H_{2} & :y-c(t-\tau)\sin\theta=a_{2}, & H_{3} & :x\sin\theta-y\cos\theta=a_{1}\sin\theta-a_{2}\cos\theta,\label{qqqq}\\
H_{4} & :x+c(t-\tau)\sin\theta=b_{1}, & H_{5} & :y-c(t-\tau)\cos\theta=a_{2}, & H_{6} & :x\cos\theta+y\sin\theta=b_{1}\cos\theta+a_{2}\sin\theta,\\
H_{7} & :x-c(t-\tau)\sin\theta=a_{1}, & H_{8} & :y+c(t-\tau)\cos\theta=b_{2}, & H_{9} & :x\cos\theta+y\sin\theta=a_{1}\cos\theta+b_{2}\sin\theta,\\
H_{10} & :x+c(t-\tau)\cos\theta=b_{1}, & H_{11} & :y+c(t-\tau)\sin\theta=b_{2}, & H_{12} & :x\sin\theta-y\cos\theta=b_{1}\sin\theta-b_{2}\cos\theta.\label{zzzzz}
\end{align}
}Finally, we consider the space $\mathscr{E}_{\mathcal{H}}$ as defined
in the notation~ \ref{kqqkkd}.

\subsection{Main results}

We prove that the analogous of the theorems 2.1, 2.2 and 2.3 in the
paper \cite{sob almeida 2026 arxiv 1} for the four velocity Broadwell
model also hold for the general Broadwell four velocity model; namely 
\begin{thm}
\label{thm:Suppose-.-Then-1-1-1}For sufficiently large $\sigma$,
suppose that $\mathfrak{q}<f_{\sigma}\left(R_{0}\right)$ and $\tau'-\tau\leq g\left(\mathfrak{q}\right)$.
Then we have $\mathfrak{p}_{\sigma}\mathfrak{q}_{\sigma}\leq\dfrac{1}{4}$
and the problem $\mathcal{P}_{\tau,\tau'}$ possesses a unique non-negative
solution $N=\left(N_{i}\right)^{4}_{i=1}\in C\left(K_{\tau,\tau'};\R\right)^{4}$
such that the derivatives $\dfrac{\partial N_{i}}{\partial t},\dfrac{\partial N_{i}}{\partial x},\dfrac{\partial N_{i}}{\partial y},$
(for $i=1,2,3,4$) are defined everywhere on $\operatorname{int}(K_{\tau,\tau'})\setminus\bigcup^{12}_{m=1}H_{m}$,
where they are continuous and bounded and verify 
\begin{equation}
{\displaystyle \max_{1\leq i\leq4}\|N_{i}\|_{1}}  \leq\dfrac{1-\sqrt{1-4\mathfrak{p}_{\sigma}\mathfrak{q}_{\sigma}}}{2\mathfrak{p}_{\sigma}}.\label{eq:loosqz-1-1-1}
\end{equation}
In particular, if $R_{0}<\varPsi\left(\sigma\right)$ with $\mathfrak{q}<f_{\sigma}(R_{0})$
and $\tau'-\tau=g(\mathfrak{q})$ then \eqref{eq:loosqz-1-1-1} implies
that 
\begin{equation}
\gamma\cdot\max_{1\leq i\leq4}\|N_{i}\|_{1}<f_{\sigma}(R_{0}).
\end{equation}
\end{thm}

\begin{thm}
\label{thm:Suppose-.-Then-1-2} For a sufficiently large $\sigma,$
suppose that $R_{0}<\varPsi\left(\sigma\right)$ and 
\begin{multline*}
{\displaystyle \max_{1\leq i\leq4}}\left\{ \right.\left\Vert N^{0}_{i}\right\Vert _{1}\left\Vert N^{-}_{1}\right\Vert _{1},\left\Vert N^{--}_{1}\right\Vert _{1},\left\Vert N^{+}_{2}\right\Vert _{1},\left\Vert N^{--}_{2}\right\Vert _{1},\\
\left\Vert N^{-}_{3}\right\Vert _{1},\left\Vert N^{++}_{3}\right\Vert _{1},\left\Vert N^{+}_{4}\right\Vert _{1},\left\Vert N^{++}_{4}\right\Vert _{1}\left.\right\} \leq R_{0}.
\end{multline*}
Then the problem $\mathcal{P}$ admits a unique non-negative, continuous,
and bounded global solution $N=\left(N_{i}\right)^{4}_{i=1}$ with
bounded derivatives satisfying ${\displaystyle \max_{1\leq i\leq4}}\left\Vert N_{i}\right\Vert _{1}\leq R_{0}.$ 
\end{thm}

\section{Integral Operators and Fixed Point Formulation }\label{sec:Integral-Operators-and}

\subsection{Integral operators}
\begin{notation}
Let $f_{1}$ and $f_{2}$ be two real functions defined on the compact
$K_{\tau,\tau'}.$ The identity function of the set $\left\{ \left(t,x,y\right)\in K_{\tau,\tau'}:f_{1}\left(t,x,y\right)\leq0\text{ and }f_{2}\left(t,x,y\right)\leq0\right\} $
is denoted by $\mathbb{I}_{\begin{cases}
f_{1}\left(t,x,y\right)\leq0\\
f_{2}\left(t,x,y\right)\leq0
\end{cases}}.$ 
\end{notation}

\begin{defn}
We introduce the integral operator, denoted by $\mathcal{T}_{\tau,\tau'}$,
which maps any function $M=(M_{1},M_{2},M_{3},M_{4})\in C(K_{\tau,\tau'};\mathbb{R})^{4}$
to its image $\mathcal{T}_{\tau,\tau'}M=\left(\left(\mathcal{T}_{\tau,\tau'}M\right)_{i}\right)^{4}_{i=1}$.
For any $(t,x,y)\in K_{\tau,\tau'}$, the components of this image
are explicitly defined by formulas \eqref{eq:oloooso-1} through \eqref{eq:olole-1-1-1-1-1-2-1-1}:
\begin{multline}
\left(\mathcal{T}_{\tau,\tau'}M\right)_{1}\left(t,x,y\right)=\left(\mathcal{T}_{\tau,\tau'}M\right)^{A}_{1}\left(t,x,y\right)\cdot\mathbb{I}_{\begin{cases}
x-c\left(t-\tau\right)\cos\theta\geq a_{1}\\
y-c\left(t-\tau\right)\sin\theta\geq a_{2}
\end{cases}}\left(t,x,y\right)+\\
\left(\mathcal{T}_{\tau,\tau'}M\right)^{B}_{1}\left(t,x,y\right)\cdot\mathbb{I}_{\begin{cases}
x-c\left(t-\tau\right)\cos\theta\leq a_{1}\\
x\sin\theta-y\cos\theta\leq a_{1}\sin\theta-a_{2}\cos\theta
\end{cases}}\left(t,x,y\right)+\\
\left(\mathcal{T}_{\tau,\tau'}M\right)^{C}_{1}\left(t,x,y\right)\cdot\mathbb{I}_{\begin{cases}
y-c\left(t-\tau\right)\sin\theta\leq a_{2}\\
x\sin\theta-y\cos\theta\geq a_{1}\sin\theta-a_{2}\cos\theta
\end{cases}}\left(t,x,y\right)\label{eq:oloooso-1}
\end{multline}
\begin{multline}
\left(\mathcal{T}_{\tau,\tau'}M\right)^{A}_{1}\left(t,x,y\right)={\displaystyle \int^{t}_{\tau}}Q\left(M\right)\left(\right.s,x+c\left(s-t\right)\cos\theta,\\
y+c\left(s-t\right)\sin\theta\left.\right)ds+\overline{N^{\tau}_{1}}\left(t,x,y\right)\label{eq:uiepom-1}
\end{multline}
\begin{multline}
\left(\mathcal{T}_{\tau,\tau'}M\right)^{B}_{1}\left(t,x,y\right)={\textstyle {\displaystyle \int^{t}_{t-\frac{1}{c\cos\theta}x+\frac{a_{1}}{c\cos\theta}}}}{\textstyle Q\left(M\right)}\left(\right.s,\\
x+c\left(s-t\right)\cos\theta,y+c\left(s-t\right)\sin\theta\left.\right)ds+\overline{N^{-}_{1}}\left(t,x,y\right);\label{opiozp-1}
\end{multline}
\begin{multline}
\left(\mathcal{T}_{\tau,\tau'}M\right)^{C}_{1}\left(t,x,y\right)={\textstyle {\displaystyle \int^{t}_{t-\frac{1}{c\sin\theta}y+\frac{a_{2}}{c\sin\theta}}}}{\textstyle Q\left(M\right)}\left(\right.s,\\
x+c\left(s-t\right)\cos\theta,y+c\left(s-t\right)\sin\theta\left.\right)ds+\overline{N^{--}_{1}}\left(t,x,y\right).\label{eq:looaoolzo-1}
\end{multline}
\begin{multline}
\left(\mathcal{T}_{\tau,\tau'}M\right)_{2}\left(t,x,y\right)=\left(\mathcal{T}_{\tau,\tau'}M\right)^{A}_{2}\left(t,x,y\right)\cdot\mathbb{I}_{\begin{cases}
x+c\left(t-\tau\right)\sin\theta\leq b_{1}\\
y-c\left(t-\tau\right)\cos\theta\geq a_{2}
\end{cases}}\left(t,x,y\right)+\\
\left(\mathcal{T}_{\tau,\tau'}M\right)^{B}_{2}\left(t,x,y\right)\cdot\mathbb{I}_{\begin{cases}
x+c\left(t-\tau\right)\sin\theta\geq b_{1}\\
x\cos\theta+y\sin\theta\geq b_{1}\cos\theta+a_{2}\sin\theta
\end{cases}}\left(t,x,y\right)+\\
\left(\mathcal{T}_{\tau,\tau'}M\right)^{C}_{2}\left(t,x,y\right)\cdot\mathbb{I}_{\begin{cases}
y-c\left(t-\tau\right)\cos\theta\leq a_{2}\\
x\cos\theta+y\sin\theta\leq b_{1}\cos\theta+a_{2}\sin\theta
\end{cases}}\left(t,x,y\right)\label{eq:ppmms-2-1-3-1}
\end{multline}

\begin{multline}
\left(\mathcal{T}_{\tau,\tau'}M\right)^{A}_{2}\left(t,x,y\right)={\textstyle {\displaystyle \int^{t}_{\tau}}}-Q\left(M\right)\left(\right.s,x-c\left(s-t\right)\sin\theta,\\
y+c\left(s-t\right)\cos\theta\left.\right)ds+\overline{N^{\tau}_{2}}\left(t,x,y\right)\label{eq:zoppz-1}
\end{multline}

\begin{multline}
\left(\mathcal{T}_{\tau,\tau'}M\right)^{B}_{2}\left(t,x,y\right)={\textstyle {\displaystyle \int^{t}_{t+\frac{1}{c\sin\theta}x-\frac{b_{1}}{c\sin\theta}}}}-Q\left(M\right)\left(\right.s,x-c\left(s-t\right)\sin\theta,\\
y+c\left(s-t\right)\cos\theta\left.\right)ds+{\textstyle \overline{N^{+}_{2}}\left(t,x,y\right);}\label{eq:opps-1}
\end{multline}

\begin{multline}
\left(\mathcal{T}_{\tau,\tau'}M\right)^{C}_{2}\left(t,x,y\right)={\textstyle {\displaystyle \int^{t}_{t-\frac{1}{c\cos\theta}y+\frac{a_{2}}{c\cos\theta}}}}-Q\left(M\right)\left(\right.s,x-c\left(s-t\right)\sin\theta,\\
y+c\left(s-t\right)\cos\theta\left.\right)ds+\overline{N^{--}_{2}}\left(t,x,y\right);\label{eq:opos-1}
\end{multline}

\begin{multline}
\left(\mathcal{T}_{\tau,\tau'}M\right)_{3}\left(t,x,y\right)=\left(\mathcal{T}_{\tau,\tau'}M\right)^{A}_{3}\left(t,x,y\right)\cdot\mathbb{I}_{\begin{cases}
x-c\left(t-\tau\right)\sin\theta\geq a_{1}\\
y+c\left(t-\tau\right)\cos\theta\leq b_{2}
\end{cases}}\left(t,x,y\right)+\\
\left(\mathcal{T}_{\tau,\tau'}M\right)^{B}_{3}\left(t,x,y\right)\cdot\mathbb{I}_{\begin{cases}
x-c\left(t-\tau\right)\sin\theta\leq a_{1}\\
x\cos\theta+y\sin\theta\leq a_{1}\cos\theta+b_{2}\sin\theta
\end{cases}}\left(t,x,y\right)+\\
\left(\mathcal{T}_{\tau,\tau'}M\right)^{C}_{3}\left(t,x,y\right)\cdot\mathbb{I}_{\begin{cases}
y+c\left(t-\tau\right)\cos\theta\geq b_{2}\\
x\cos\theta+y\sin\theta\geq a_{1}\cos\theta+b_{2}\sin\theta
\end{cases}}\left(t,x,y\right)\label{aoalal-1-1-2-1}
\end{multline}
\begin{multline}
\left(\mathcal{T}_{\tau,\tau'}M\right)^{A}_{3}\left(t,x,y\right)={\displaystyle \int^{t}_{\tau}}-Q\left(M\right)\left(s,x+c\left(s-t\right)\sin\theta,y-c\left(s-t\right)\cos\theta\right)ds+\\
\overline{N^{\tau}_{3}}\left(t,x,y\right)\label{eq:lqooqsd-1}
\end{multline}

\begin{multline}
\left(\mathcal{T}_{\tau,\tau'}M\right)^{B}_{3}\left(t,x,y\right)=\\
{\displaystyle \int^{t}_{t-\frac{1}{c\sin\theta}x+\frac{a_{1}}{c\sin\theta}}}-Q\left(M\right)\left(s,x+c\left(s-t\right)\sin\theta,y-c\left(s-t\right)\cos\theta\right)ds+\\
{\textstyle \overline{N^{-}_{3}}\left(t,x,y\right);}\label{eq:kosil-1}
\end{multline}

\begin{multline}
\left(\mathcal{T}_{\tau,\tau'}M\right)^{C}_{3}\left(t,x,y\right)=\\
{\displaystyle \int^{t}_{t+\frac{1}{c\cos\theta}y-\frac{b_{2}}{c\cos\theta}}}-Q\left(M\right)\left(s,x+c\left(s-t\right)\sin\theta,y-c\left(s-t\right)\cos\theta\right)ds+\\
{\textstyle \overline{N^{++}_{3}}\left(t,x,y\right)}\label{eq:loiuij-1}
\end{multline}

\begin{multline}
\left(\mathcal{T}_{\tau,\tau'}M\right)_{4}\left(t,x,y\right)=\left(\mathcal{T}_{\tau,\tau'}M\right)^{A}_{4}\left(t,x,y\right)\cdot\mathbb{I}_{\begin{cases}
x+c\left(t-\tau\right)\cos\theta\leq b_{1}\\
y+c\left(t-\tau\right)\sin\theta\leq b_{2}
\end{cases}}\left(t,x,y\right)+\\
\left(\mathcal{T}_{\tau,\tau'}M\right)^{B}_{4}\left(t,x,y\right)\cdot\mathbb{I}_{\begin{cases}
x+c\left(t-\tau\right)\cos\theta\geq b_{1}\\
x\sin\theta-y\cos\theta\geq b_{1}\sin\theta-b_{2}\cos\theta
\end{cases}}\left(t,x,y\right)+\\
\left(\mathcal{T}_{\tau,\tau'}M\right)^{C}_{4}\left(M\right)\left(t,x,y\right)\cdot\mathbb{I}_{\begin{cases}
y+c\left(t-\tau\right)\sin\theta\geq b_{2}\\
x\sin\theta-y\cos\theta\leq b_{1}\sin\theta-b_{2}\cos\theta
\end{cases}}\left(t,x,y\right)\label{aoalal-1-1-1-2-1}
\end{multline}

\begin{multline}
\left(\mathcal{T}_{\tau,\tau'}M\right)^{A}_{4}\left(t,x,y\right)={\displaystyle \int^{t}_{0}}Q\left(M\right)\left(s,x-c\left(s-t\right)\cos\theta,y-c\left(s-t\right)\sin\theta\right)ds+\\
\overline{N^{\tau}_{4}}\left(t,x,y\right)\label{eq:olole-1-2-1-1-1}
\end{multline}

\begin{multline}
\left(\mathcal{T}_{\tau,\tau'}M\right)^{B}_{4}\left(t,x,y\right)=\\
{\displaystyle \int^{t}_{t+\frac{1}{c\cos\theta}x-\frac{b_{1}}{c\cos\theta}}}Q\left(M\right)\left(s,x-c\left(s-t\right)\cos\theta,y-c\left(s-t\right)\sin\theta\right)ds+\\
{\textstyle \overline{N^{+}_{4}}\left(t,x,y\right);}\label{eq:olole-1-1-1-2-1-1}
\end{multline}

\begin{multline}
\left(\mathcal{T}_{\tau,\tau'}M\right)^{C}_{4}\left(t,x,y\right)=\\
{\displaystyle \int^{t}_{t+\frac{1}{c\sin\theta}y-\frac{b_{2}}{c\sin\theta}}}Q\left(M\right)\left(s,x-c\left(s-t\right)\cos\theta,y-c\left(s-t\right)\sin\theta\right)ds+\\
{\textstyle \overline{N^{++}_{4}}\left(t,x,y\right).}\label{eq:olole-1-1-1-1-1-2-1-1}
\end{multline}
\end{defn}

\begin{notation}
\label{nnooauu}Furthermore, consider for $\sigma>0$ and any function
$M=(M_{1},M_{2},M_{3},M_{4})\in C(K_{\tau,\tau'};\mathbb{R})^{4}$,
the functions $|M|=\left(|M_{1}|,|M_{2}|,|M_{3}|,|M_{4}|\right)$,
$\rho(M)=\sum^{4}_{i=1}M_{i}$ and 
\begin{equation}
\begin{cases}
Q^{\sigma}_{1}(M)=\sigma\rho(M)M_{1}+Q(M),\\
Q^{\sigma}_{2}(M)=\sigma\rho(M)M_{2}-Q(M),\\
Q^{\sigma}_{3}(M)=\sigma\rho(M)M_{3}-Q(M),\\
Q^{\sigma}_{4}(M)=\sigma\rho(M)M_{4}+Q(M).
\end{cases}\label{pzlooo}
\end{equation}
introduced in the Section 3 of \cite{sob almeida 2026 arxiv 1}. It
is shown there that $Q^{\sigma}_{i}\left(\left|M\right|\right)\geq0$
for all $M$ and all $i\in\{1,2,3,4\}$, provided that $\sigma\geq2cS$. 
\end{notation}

\begin{defn}
For any $\sigma>0$, we introduce the integral operator denoted by
$\mathcal{T}^{\sigma}_{\tau,\tau'}$, which maps any function $M=(M_{1},M_{2},M_{3},M_{4})\in C(K_{\tau,\tau'};\mathbb{R})^{4}$
to its image $\mathcal{T}^{\sigma}_{\tau,\tau'}M=\left(\left(\mathcal{T}^{\sigma}_{\tau,\tau'}M\right)_{i}\right)^{4}_{i=1}$.
For any $(t,x,y)\in K_{\tau,\tau'}$, the components of this image
are explicitly defined by formulas \eqref{eq:oloooso-1-1} through
\eqref{eq:olole-1-1-1-1-1-2-2}: 
\begin{multline}
\left(\mathcal{T}^{\sigma}_{\tau,\tau'}M\right)_{1}\left(t,x,y\right)=\left(\mathcal{T}^{\sigma}_{\tau,\tau'}M\right)^{A}_{1}\left(t,x,y\right)\cdot\mathbb{I}_{\begin{cases}
x-c\left(t-\tau\right)\cos\theta\geq a_{1}\\
y-c\left(t-\tau\right)\sin\theta\geq a_{2}
\end{cases}}\left(t,x,y\right)+\\
\left(\mathcal{T}^{\sigma}_{\tau,\tau'}M\right)^{B}_{1}\left(t,x,y\right)\cdot\mathbb{I}_{\begin{cases}
x-c\left(t-\tau\right)\cos\theta\leq a_{1}\\
x\sin\theta-y\cos\theta\leq a_{1}\sin\theta-a_{2}\cos\theta
\end{cases}}\left(t,x,y\right)+\\
\left(\mathcal{T}^{\sigma}_{\tau,\tau'}M\right)^{C}_{1}\left(t,x,y\right)\cdot\mathbb{I}_{\begin{cases}
y-c\left(t-\tau\right)\sin\theta\leq a_{2}\\
x\sin\theta-y\cos\theta\geq a_{1}\sin\theta-a_{2}\cos\theta
\end{cases}}\left(t,x,y\right)\label{eq:oloooso-1-1}
\end{multline}
where 
\begin{multline}
\left(\mathcal{T}^{\sigma}_{\tau,\tau'}M\right)^{A}_{1}\left(t,x,y\right)=\Biggl({\displaystyle \int^{t}_{\tau}}e^{\sigma\int^{s}_{\tau}\rho\left(\left|M\right|\right)\left(r,x+c\left(r-t\right)\cos\theta,y+c\left(r-t\right)\sin\theta\right)dr}\cdot Q^{\sigma}_{1}\left(\left|M\right|\right)\\
\left(s,x+c\left(s-t\right)\cos\theta,y+c\left(s-t\right)\sin\theta\right)ds+\overline{N^{\tau}_{1}}\left(t,x,y\right)\Biggr)\cdot\\
e^{-\sigma{\textstyle {\displaystyle \int^{t}_{\tau}}}\rho\left(\left|M\right|\right)\left(s,x+c\left(s-t\right)\cos\theta,y+c\left(s-t\right)\sin\theta\right)ds}.\label{eq:olole-1-3}
\end{multline}

\begin{multline}
\left(\mathcal{T}^{\sigma}_{\tau,\tau'}M\right)^{B}_{1}\left(t,x,y\right)=\\
\Biggl({\textstyle {\displaystyle \int^{t}_{t-\frac{1}{c\cos\theta}x+\frac{a_{1}}{c\cos\theta}}}e^{\sigma\int^{s}_{t-\frac{1}{c\cos\theta}x+\frac{a_{1}}{c\cos\theta}}\rho\left(\left|M\right|\right)\left(r,x+c\left(r-t\right)\cos\theta,y+c\left(r-t\right)\sin\theta\right)dr}\cdot Q^{\sigma}_{1}\left(\left|M\right|\right)}\\
\left(s,x+c\left(s-t\right)\cos\theta,y+c\left(s-t\right)\sin\theta\right)ds+\overline{N^{-}_{1}}\left(t,x,y\right){\textstyle \Biggr)}\cdot\\
{\textstyle e^{-\sigma{\displaystyle \int^{t}_{t-\frac{1}{c\cos\theta}x+\frac{a_{1}}{c\cos\theta}}}\rho\left(\left|M\right|\right)\left(s,x+c\left(s-t\right)\cos\theta,y+c\left(s-t\right)\sin\theta\right)ds}}.\label{eq:olole-1-1-2}
\end{multline}
\begin{multline}
\left(\mathcal{T}^{\sigma}_{\tau,\tau'}M\right)^{C}_{1}\left(t,x,y\right)=\\
\Biggl({\textstyle {\displaystyle \int^{t}_{t-\frac{1}{c\sin\theta}y+\frac{a_{2}}{c\sin\theta}}}e^{\sigma\int^{s}_{t-\frac{1}{c\sin\theta}y+\frac{a_{2}}{c\sin\theta}}\rho\left(\left|M\right|\right)\left(r,x+c\left(r-t\right)\cos\theta,y+c\left(r-t\right)\sin\theta\right)dr}}\cdot Q^{\sigma}_{1}\left(\left|M\right|\right)\\
{\textstyle \left(s,x+c\left(s-t\right)\cos\theta,y+c\left(s-t\right)\sin\theta\right)ds}+{\textstyle \overline{N^{--}_{1}}\left(t,x,y\right)\Biggr)}\cdot\\
{\textstyle e^{-\sigma{\displaystyle \int^{t}_{t-\frac{1}{c\sin\theta}y+\frac{a_{2}}{c\sin\theta}}}\rho\left(\left|M\right|\right)\left(s,x+c\left(s-t\right)\cos\theta,y+c\left(s-t\right)\sin\theta\right)ds}}.\label{eq:olole-1-1-1-1-3}
\end{multline}
\begin{multline}
\left(\mathcal{T}^{\sigma}_{\tau,\tau'}M\right)_{2}\left(t,x,y\right)=\left(\mathcal{T}^{\sigma}_{\tau,\tau'}M\right)^{A}_{2}\left(t,x,y\right)\cdot\mathbb{I}_{\begin{cases}
x+c\left(t-\tau\right)\sin\theta\leq b_{1}\\
y-c\left(t-\tau\right)\cos\theta\geq a_{2}
\end{cases}}\left(t,x,y\right)+\\
\left(\mathcal{T}^{\sigma}_{\tau,\tau'}M\right)^{B}_{2}\left(t,x,y\right)\cdot\mathbb{I}_{\begin{cases}
x+c\left(t-\tau\right)\sin\theta\geq b_{1}\\
x\cos\theta+y\sin\theta\geq b_{1}\cos\theta+a_{2}\sin\theta
\end{cases}}\left(t,x,y\right)+\\
\left(\mathcal{T}^{\sigma}_{\tau,\tau'}M\right)^{C}_{2}\left(t,x,y\right)\cdot\mathbb{I}_{\begin{cases}
y-c\left(t-\tau\right)\cos\theta\leq a_{2}\\
x\cos\theta+y\sin\theta\leq b_{1}\cos\theta+a_{2}\sin\theta
\end{cases}}\left(t,x,y\right)\label{eq:ppmms-2-1-3-1-1}
\end{multline}
where 
\begin{multline}
\left(\mathcal{T}^{\sigma}_{\tau,\tau'}M\right)^{A}_{2}\left(t,x,y\right)=\Biggl({\textstyle {\displaystyle \int^{t}_{\tau}}e^{\sigma\int^{s}_{\tau}\rho\left(\left|M\right|\right)\left(r,x-c\left(r-t\right)\sin\theta,y+c\left(r-t\right)\cos\theta\right)dr}\cdot}\\
{\textstyle Q^{\sigma}_{2}\left(\left|M\right|\right)\left(s,x-c\left(s-t\right)\sin\theta,y+c\left(s-t\right)\cos\theta\right)ds}+{\textstyle \overline{N^{\tau}_{2}}\left(t,x,y\right)\Biggr)}\cdot\\
{\textstyle e^{-\sigma{\displaystyle \int^{t}_{\tau}}\rho\left(\left|M\right|\right)\left(s,x-c\left(s-t\right)\sin\theta,y+c\left(s-t\right)\cos\theta\right)ds}}.\label{eq:olole-1-1-1-1-1-1-1}
\end{multline}
\begin{multline}
\left(\mathcal{T}^{\sigma}_{\tau,\tau'}M\right)^{B}_{2}\left(t,x,y\right)=\\
\Biggl({\textstyle {\displaystyle \int^{t}_{t+\frac{1}{c\sin\theta}x-\frac{b_{1}}{c\sin\theta}}}e^{\sigma\int^{s}_{t+\frac{1}{c\sin\theta}x-\frac{b_{1}}{c\sin\theta}}\rho\left(\left|M\right|\right)\left(r,x-c\left(r-t\right)\sin\theta,y+c\left(r-t\right)\cos\theta\right)dr}}\cdot\\
{\textstyle Q^{\sigma}_{2}\left(\left|M\right|\right)\left(s,x-c\left(s-t\right)\sin\theta,y+c\left(s-t\right)\cos\theta\right)ds}+{\textstyle \overline{N^{+}_{2}}\left(t,x,y\right)\Biggr)}\cdot\\
{\textstyle e^{-\sigma{\displaystyle \int^{t}_{t+\frac{1}{c\sin\theta}x-\frac{b_{1}}{c\sin\theta}}}\rho\left(\left|M\right|\right)\left(s,x-c\left(s-t\right)\sin\theta,y+c\left(s-t\right)\cos\theta\right)ds}}.\label{eq:olole-1-1-1-1-2-2}
\end{multline}

\begin{multline}
\left(\mathcal{T}^{\sigma}_{\tau,\tau'}M\right)^{C}_{2}\left(t,x,y\right)=\\
\Biggl({\textstyle {\displaystyle \int^{t}_{t-\frac{1}{c\cos\theta}y+\frac{a_{2}}{c\cos\theta}}}e^{\sigma\int^{s}_{t-\frac{1}{c\cos\theta}y+\frac{a_{2}}{c\cos\theta}}\rho\left(\left|M\right|\right)\left(r,x-c\left(r-t\right)\sin\theta,y+c\left(r-t\right)\cos\theta\right)dr}}\cdot\\
{\textstyle Q^{\sigma}_{2}\left(\left|M\right|\right)\left(s,x-c\left(s-t\right)\sin\theta,y+c\left(s-t\right)\cos\theta\right)ds}+{\textstyle \overline{N^{--}_{2}}\left(t,x,y\right)\Biggr)}\cdot\\
{\textstyle e^{-\sigma{\displaystyle \int^{t}_{t-\frac{1}{c\cos\theta}y+\frac{a_{2}}{c\cos\theta}}}\rho\left(\left|M\right|\right)\left(s,x-c\left(s-t\right)\sin\theta,y+c\left(s-t\right)\cos\theta\right)ds}}.\label{eq:olole-1-1-1-1-2-1-1}
\end{multline}

\begin{multline}
\left(\mathcal{T}^{\sigma}_{\tau,\tau'}M\right)_{3}\left(t,x,y\right)=\left(\mathcal{T}^{\sigma}_{\tau,\tau'}M\right)^{A}_{3}\left(t,x,y\right)\cdot\mathbb{I}_{\begin{cases}
x-c\left(t-\tau\right)\sin\theta\geq a_{1}\\
y+c\left(t-\tau\right)\cos\theta\leq b_{2}
\end{cases}}\left(t,x,y\right)+\\
\left(\mathcal{T}^{\sigma}_{\tau,\tau'}M\right)^{B}_{3}\left(t,x,y\right)\cdot\mathbb{I}_{\begin{cases}
x-c\left(t-\tau\right)\sin\theta\leq a_{1}\\
x\cos\theta+y\sin\theta\leq a_{1}\cos\theta+b_{2}\sin\theta
\end{cases}}\left(t,x,y\right)+\\
\left(\mathcal{T}^{\sigma}_{\tau,\tau'}M\right)^{C}_{3}\left(t,x,y\right)\cdot\mathbb{I}_{\begin{cases}
y+c\left(t-\tau\right)\cos\theta\geq b_{2}\\
x\cos\theta+y\sin\theta\geq a_{1}\cos\theta+b_{2}\sin\theta
\end{cases}}\left(t,x,y\right)\label{aoalal-1-1-2-1-1}
\end{multline}
where

\begin{multline}
\left(\mathcal{T}^{\sigma}_{\tau,\tau'}M\right)^{A}_{3}\left(t,x,y\right)=\Biggl({\displaystyle \int^{t}_{\tau}}e^{\sigma\int^{s}_{\tau}\rho\left(\left|M\right|\right)\left(r,x+c\left(r-t\right)\sin\theta,y-c\left(r-t\right)\cos\theta\right)dr}\cdot Q^{\sigma}_{3}\left(\left|M\right|\right)\\
\left(s,x+c\left(s-t\right)\sin\theta,y-c\left(s-t\right)\cos\theta\right)ds+\overline{N^{\tau}_{3}}\left(t,x,y\right)\Biggr)\cdot\\
e^{-\sigma{\textstyle {\displaystyle \int^{t}_{\tau}}}\rho\left(\left|M\right|\right)\left(s,x+c\left(s-t\right)\sin\theta,y-c\left(s-t\right)\cos\theta\right)ds}.\label{eq:olole-1-2-2}
\end{multline}
\begin{multline}
\left(\mathcal{T}^{\sigma}_{\tau,\tau'}M\right)^{B}_{3}\left(t,x,y\right)=\\
\Biggl({\textstyle {\displaystyle \int^{t}_{t-\frac{1}{c\sin\theta}x+\frac{a_{1}}{c\sin\theta}}}e^{\sigma\int^{s}_{t-\frac{1}{c\sin\theta}x+\frac{a_{1}}{c\sin\theta}}\rho\left(\left|M\right|\right)\left(r,x+c\left(r-t\right)\sin\theta,y-c\left(r-t\right)\cos\theta\right)dr}}\cdot\\
{\textstyle Q^{\sigma}_{3}\left(\left|M\right|\right)\left(s,x+c\left(s-t\right)\sin\theta,y-c\left(s-t\right)\cos\theta\right)ds}+{\textstyle \overline{N^{-}_{3}}\left(t,x,y\right)\biggr)}\cdot\\
{\textstyle e^{-\sigma{\displaystyle \int^{t}_{t-\frac{1}{c\sin\theta}x+\frac{a_{1}}{c\sin\theta}}}\rho\left(\left|M\right|\right)\left(s,x+c\left(s-t\right)\sin\theta,y-c\left(s-t\right)\cos\theta\right)ds}}.\label{eq:olole-1-1-1-3}
\end{multline}

\begin{multline}
\left(\mathcal{T}^{\sigma}_{\tau,\tau'}M\right)^{C}_{3}\left(t,x,y\right)=\\
\Biggl({\textstyle {\displaystyle \int^{t}_{t+\frac{1}{c\cos\theta}y-\frac{b_{2}}{c\cos\theta}}}e^{\sigma\int^{s}_{t+\frac{1}{c\cos\theta}y-\frac{b_{2}}{c\cos\theta}}\rho\left(\left|M\right|\right)\left(r,x+c\left(r-t\right)\sin\theta,y-c\left(r-t\right)\cos\theta\right)dr}}\cdot\\
{\textstyle Q^{\sigma}_{3}\left(\left|M\right|\right)\left(s,x+c\left(s-t\right)\sin\theta,y-c\left(s-t\right)\cos\theta\right)ds}+{\textstyle \overline{N^{++}_{3}}\left(t,x,y\right)\Biggr)}\cdot\\
{\textstyle e^{-\sigma{\displaystyle \int^{t}_{t+\frac{1}{c\cos\theta}y-\frac{b_{2}}{c\cos\theta}}}\rho\left(\left|M\right|\right)\left(s,x+c\left(s-t\right)\sin\theta,y-c\left(s-t\right)\cos\theta\right)ds}}.\label{eq:olole-1-1-1-1-1-3}
\end{multline}

\begin{multline}
\left(\mathcal{T}^{\sigma}_{\tau,\tau'}M\right)_{4}\left(t,x,y\right)=\left(\mathcal{T}^{\sigma}_{\tau,\tau'}M\right)^{A}_{4}\left(t,x,y\right)\cdot\mathbb{I}_{\begin{cases}
x+c\left(t-\tau\right)\cos\theta\leq b_{1}\\
y+c\left(t-\tau\right)\sin\theta\leq b_{2}
\end{cases}}\left(t,x,y\right)+\\
\left(\mathcal{T}^{\sigma}_{\tau,\tau'}M\right)^{B}_{4}\left(t,x,y\right)\cdot\mathbb{I}_{\begin{cases}
x+c\left(t-\tau\right)\cos\theta\geq b_{1}\\
x\sin\theta-y\cos\theta\geq b_{1}\sin\theta-b_{2}\cos\theta
\end{cases}}\left(t,x,y\right)+\\
\left(\mathcal{T}^{\sigma}_{\tau,\tau'}M\right)^{C}_{4}\left(M\right)\left(t,x,y\right)\cdot\mathbb{I}_{\begin{cases}
y+c\left(t-\tau\right)\sin\theta\geq b_{2}\\
x\sin\theta-y\cos\theta\leq b_{1}\sin\theta-b_{2}\cos\theta
\end{cases}}\left(t,x,y\right)\label{aoalal-1-1-1-2-1-1}
\end{multline}
where

\begin{multline}
\left(\mathcal{T}^{\sigma}_{\tau,\tau'}M\right)^{A}_{4}\left(t,x,y\right)=\Biggl({\displaystyle \int^{t}_{\tau}}e^{\sigma\int^{s}_{\tau}\rho\left(\left|M\right|\right)\left(r,x-c\left(r-t\right)\cos\theta,y-c\left(r-t\right)\sin\theta\right)dr}\cdot Q^{\sigma}_{4}\left(\left|M\right|\right)\\
\left(s,x-c\left(s-t\right)\cos\theta,y-c\left(s-t\right)\sin\theta\right)ds+\overline{N^{\tau}_{4}}\left(t,x,y\right)\Biggr)\cdot\\
e^{-\sigma{\textstyle {\displaystyle \int^{t}_{\tau}}}\rho\left(\left|M\right|\right)\left(s,x-c\left(s-t\right)\cos\theta,y-c\left(s-t\right)\sin\theta\right)ds}.\label{eq:olole-1-2-1-2}
\end{multline}
\begin{multline}
\left(\mathcal{T}^{\sigma}_{\tau,\tau'}M\right)^{B}_{4}\left(t,x,y\right)=\\
\Biggl({\textstyle {\displaystyle \int^{t}_{t+\frac{1}{c\cos\theta}x-\frac{b_{1}}{c\cos\theta}}}e^{\sigma\int^{s}_{t+\frac{1}{c\cos\theta}x-\frac{b_{1}}{c\cos\theta}}\rho\left(\left|M\right|\right)\left(r,x-c\left(r-t\right)\cos\theta,y-c\left(r-t\right)\sin\theta\right)dr}}\cdot\\
{\textstyle Q^{\sigma}_{4}\left(\left|M\right|\right)\left(s,x-c\left(s-t\right)\cos\theta,y-c\left(s-t\right)\sin\theta\right)ds}+{\textstyle \overline{N^{+}_{4}}\left(t,x,y\right)\Biggr)}\cdot\\
{\textstyle e^{-\sigma{\displaystyle \int^{t}_{t+\frac{1}{c\cos\theta}x-\frac{b_{1}}{c\cos\theta}}}\rho\left(\left|M\right|\right)\left(s,x-c\left(s-t\right)\cos\theta,y-c\left(s-t\right)\sin\theta\right)ds}}.\label{eq:olole-1-1-1-2-2}
\end{multline}
\begin{multline}
\left(\mathcal{T}^{\sigma}_{\tau,\tau'}M\right)^{C}_{4}\left(t,x,y\right)=\\
\Biggl({\textstyle {\displaystyle \int^{t}_{t+\frac{1}{c\sin\theta}y-\frac{b_{2}}{c\sin\theta}}}e^{\sigma\int^{s}_{t+\frac{1}{c\sin\theta}y-\frac{b_{2}}{c\sin\theta}}\rho\left(\left|M\right|\right)\left(r,x-c\left(r-t\right)\cos\theta,y-c\left(r-t\right)\sin\theta\right)dr}}\cdot\\
{\textstyle Q^{\sigma}_{4}\left(\left|M\right|\right)\left(s,x-c\left(s-t\right)\cos\theta,y-c\left(s-t\right)\sin\theta\right)ds}+{\textstyle \overline{N^{++}_{4}}\left(t,x,y\right)\Biggr)}\cdot\\
{\textstyle e^{-\sigma{\displaystyle \int^{t}_{t+\frac{1}{c\sin\theta}y-\frac{b_{2}}{c\sin\theta}}}\rho\left(\left|M\right|\right)\left(s,x-c\left(s-t\right)\cos\theta,y-c\left(s-t\right)\sin\theta\right)ds}}.\label{eq:olole-1-1-1-1-1-2-2}
\end{multline}
\end{defn}

\subsection{Fixed point formulation}
\begin{prop}
\label{aalalaooa}The integral operator $\mathcal{T}_{\tau,\tau'}$
and the family of integral operators $\{\mathcal{T}^{\sigma}_{\tau,\tau'}\}_{\sigma>0}$
defined above map the space $C(K_{\tau,\tau'};\mathbb{R})^{4}$ into
itself and satisfy the following properties:\\
 \ \ \ \ i) Any differentiable fixed point of $\mathcal{T}_{\tau,\tau'}$
is a solution to problem $\mathcal{P}_{\tau,\tau'}$, and conversely,
every solution to problem $\mathcal{P}_{\tau,\tau'}$ is a fixed point
of $\mathcal{T}_{\tau,\tau'}$.\\
 \ \ \ \ ii) There exists a constant $\sigma_{0}>0$ such that,
for all $\sigma\ge\sigma_{0}$, the operator $\mathcal{T}^{\sigma}_{\tau,\tau'}$
is non-negative and any differentiable fixed point of the operator
$\mathcal{T}^{\sigma}_{\tau,\tau'}$ is a solution to problem $\mathcal{P}_{\tau,\tau'}$. 
\end{prop}

\begin{proof}
For a fixed $M$, each component $\left(\mathcal{T}^{\sigma}_{\tau,\tau'}M\right)_{i}$
($i=1,\dots,4$) is obtained as the unique solution to a specific
equation, determined as follows: the first component $\left(\mathcal{T}^{\sigma}_{\tau,\tau'}M\right)_{1}$
satisfies 
\begin{equation}
{\textstyle \dfrac{\p N_{1}}{\p t}+c\cos\theta\dfrac{\p N_{1}}{\p x}}{\displaystyle +c\sin\theta\dfrac{\p N_{1}}{\p y}+\sigma\rho\left(\left|M\right|\right)N_{1}}=Q^{\sigma}_{1}\left(\left|M\right|\right),\;\left(t,x,y\right)\in\mathring{K}_{\tau,\tau'}\label{olopoz}
\end{equation}
subject to the conditions \eqref{eq:lsos} (for $i=1$) through \eqref{eq:mppos}.
\\
 The second component $\left(\mathcal{T}^{\sigma}_{\tau,\tau'}M\right)_{2}$
satisfies 
\[
\dfrac{\p N_{2}}{\p t}-c\sin\theta\dfrac{\p N_{2}}{\p x}+c\cos\theta\dfrac{\p N_{2}}{\p y}+\sigma\rho\left(\left|M\right|\right)N_{2}=Q^{\sigma}_{2}\left(\left|M\right|\right),\;\left(t,x,y\right)\in\mathring{K}_{\tau,\tau'}
\]
under the conditions \eqref{eq:lsos} (for $i=2$), \eqref{eq:mpsss},
\eqref{eq:lsoo-1}.\\
 The third component $\left(\mathcal{T}^{\sigma}_{\tau,\tau'}M\right)_{3}$
verifies 
\[
\dfrac{\p N_{3}}{\p t}+c\sin\theta\dfrac{\p N_{3}}{\p x}-c\cos\theta\dfrac{\p N_{3}}{\p y}+\sigma\rho\left(\left|M\right|\right)N_{3}=Q^{\sigma}_{3}\left(\left|M\right|\right),\;\left(t,x,y\right)\in\mathring{K}_{\tau,\tau'}
\]
with the conditions \eqref{eq:lsos} (for $i=3$), \eqref{ikis},
\eqref{eq:ikioi}.\\
 Finally, the fourth component $\left(\mathcal{T}^{\sigma}_{\tau,\tau'}M\right)_{4}$
is given by 
\begin{equation}
\dfrac{\p N_{4}}{\p t}-c\cos\theta\dfrac{\p N_{4}}{\p x}-c\sin\theta\dfrac{\p N_{4}}{\p y}+\sigma\rho\left(\left|M\right|\right)N_{4}=Q^{\sigma}_{4}\left(\left|M\right|\right),\;\left(t,x,y\right)\in\mathring{K}_{\tau,\tau'}\label{jzusjju}
\end{equation}
subject to \eqref{eq:lsos} (for $i=4$), \eqref{eq:ikioi-1}, \eqref{eq:oiiap}.

For a fixed $M$, the first component $\left(\mathcal{T}_{\tau,\tau'}M\right)_{1}$
is obtained as the unique solution to the equation 
\begin{equation}
\frac{\partial N_{1}}{\partial t}+c\cos\theta\frac{\partial N_{1}}{\partial x}+c\sin\theta\frac{\partial N_{1}}{\partial y}=Q(M),\quad(t,x,y)\in\mathring{K}_{\tau,\tau'},\label{ooapplao}
\end{equation}
subject to the conditions \eqref{eq:lsos} (for $i=1$) through \eqref{eq:mppos}.\\
 The second component $\left(\mathcal{T}_{\tau,\tau'}M\right)_{2}$
is the unique solution to the equation 
\begin{equation}
\dfrac{\partial N_{2}}{\partial t}-c\sin\theta\dfrac{\partial N_{2}}{\partial x}+c\cos\theta\dfrac{\partial N_{2}}{\partial y}=-Q(M),\;\left(t,x,y\right)\in\mathring{K}_{\tau,\tau'},
\end{equation}
subject to the exact same conditions imposed on $\left(\mathcal{T}^{\sigma}_{\tau,\tau'}M\right)_{2}.$\\
 Similarly, the third component $\left(\mathcal{T}_{\tau,\tau'}M\right)_{3}$
is the unique solution to the equation 
\begin{equation}
\dfrac{\partial N_{3}}{\partial t}+c\sin\theta\dfrac{\partial N_{3}}{\partial x}-c\cos\theta\dfrac{\partial N_{3}}{\partial y}=-Q(M),\;\left(t,x,y\right)\in\mathring{K}_{\tau,\tau'},
\end{equation}
subject to the exact same conditions imposed on $\left(\mathcal{T}^{\sigma}_{\tau,\tau'}M\right)_{3}.$\\
 Finally, the fourth component $\left(\mathcal{T}_{\tau,\tau'}M\right)_{4}$
is the unique solution to 
\begin{equation}
\dfrac{\partial N_{4}}{\partial t}-c\cos\theta\dfrac{\partial N_{4}}{\partial x}-c\sin\theta\dfrac{\partial N_{4}}{\partial y}=Q(M),\;\left(t,x,y\right)\in\mathring{K}_{\tau,\tau'},\label{ooppal}
\end{equation}
subject to the exact same conditions imposed on $\left(\mathcal{T}^{\sigma}_{\tau,\tau'}M\right)_{4}.$\\
 Now, if $M$ is derivable, we see from \eqref{ooapplao}-\eqref{ooppal}
with the corresponding conditions, that $\mathcal{T}_{\tau,\tau'}M=M$
if and only if $M$ is a solution to the problem $\mathcal{P}_{\tau,\tau'}$
defined as the system \eqref{eq:erter-1}-\eqref{eq:zzzfra-1} subject
to the initial and boundary conditions \eqref{eq:lsos}-\eqref{eq:oiiap}.
This completes the proof of statement (i) of the proposition.

We now show that the operator $\mathcal{T}^{\sigma}_{\tau,\tau'}$
is non-negative for $\sigma\geq2cS$. We know that $Q^{\sigma}_{i}\left(\left|M\right|\right)\geq0$
for all $M$ and all $i\in\{1,2,3,4\}$, provided that $\sigma\geq2cS$.
\\
 Furthermore, recall that $(t,x,y)\in K_{\tau,\tau'}\equiv[\tau,\tau']\times[a_{1},b_{1}]\times[a_{2},b_{2}]$
in equations \eqref{eq:oloooso-1-1}-\eqref{eq:olole-1-1-1-1-1-2-2}
which define the components of $\mathcal{T}^{\sigma}_{\tau,\tau'}M$.
It follows that the upper bound of integration in each formula is
greater than or equal to its corresponding lower bound. Since the
functions introduced in Notation ~\ref{jskskk} are non-negative
due to the non-negativity of the data, we conclude that the components
of $\mathcal{T}^{\sigma}_{\tau,\tau'}M$ are non-negative for $\sigma\geq2cS$.
Hence, for $\sigma\geq2cS$, the condition $\mathcal{T}^{\sigma}_{\tau,\tau'}M=M$
implies that $M$ is non-negative, which in turn yields $Q^{\sigma}_{i}(|M|)=Q^{\sigma}_{i}(M)$.
Consequently, by virtue of equations \eqref{olopoz} -- \eqref{ooppal}
and adopting the notation from \eqref{nnooauu}, we obtain: 
\begin{align}
\frac{\partial M_{1}}{\partial t}+c\cos\theta\frac{\partial M_{1}}{\partial x}+c\sin\theta\frac{\partial M_{1}}{\partial y} & =-\sigma\rho(M)M_{1}+Q^{\sigma}_{1}(M)=Q(M),\\
\frac{\partial M_{2}}{\partial t}-c\sin\theta\frac{\partial M_{2}}{\partial x}+c\cos\theta\frac{\partial M_{2}}{\partial y} & =-\sigma\rho(M)M_{2}+Q^{\sigma}_{2}(M)=-Q(M),\\
\frac{\partial M_{3}}{\partial t}+c\sin\theta\frac{\partial M_{3}}{\partial x}-c\cos\theta\frac{\partial M_{3}}{\partial y} & =-\sigma\rho(M)M_{3}+Q^{\sigma}_{3}(M)=-Q(M),\\
\frac{\partial M_{4}}{\partial t}-c\cos\theta\frac{\partial M_{4}}{\partial x}-c\sin\theta\frac{\partial M_{4}}{\partial y} & =-\sigma\rho(M)M_{4}+Q^{\sigma}_{4}(M)=Q(M).
\end{align}
Moreover, since $\mathcal{T}^{\sigma}_{\tau,\tau'}M$ satisfies conditions
\eqref{eq:lsos}--\eqref{eq:oiiap}, it follows from the fixed-point
relation $M=\mathcal{T}^{\sigma}_{\tau,\tau'}M$ that $M$ also satisfies
these conditions. Consequently, $M$ constitutes a solution to problem
$\mathcal{P}_{\tau,\tau'}$. This completes the proof of statement
(ii) of the proposition. 
\end{proof}

\subsection{Properties of the operators}

The operator $\mathcal{T}^{\sigma}_{\tau,\tau'}$ is continuous (
we can use arguments from \cite{sob almeida 2026 arxiv 1}, Section
4). 

It is easily seen that the estimate (5.1) in the paper \cite{sob almeida 2026 arxiv 1}
still holds for $\mathcal{T}_{\tau,\tau'}$. 
\begin{prop}
\textup{\label{vvvvv}}For all $M$ and $N:$ 
\begin{align}
\left\Vert \mathcal{T}_{\tau,\tau'}M-\mathcal{T}_{\tau,\tau'}N\right\Vert  & \leq\left(\tau'-\tau\right)\cdot4cS\left(\left\Vert M\right\Vert +\left\Vert N\right\Vert \right)\left\Vert M-N\right\Vert \label{eq:lsoopa-2-1}
\end{align}
where $\left\Vert \cdot\right\Vert $ is defined in Eq. \eqref{oosaz}. 
\end{prop}

\begin{prop}
\label{prop::kko-1-1}\textup{ Let $\mathscr{E}^{+}_{\mathcal{H}}$
denote the subset of non-negative elements of $\mathscr{E}_{\mathcal{H}}.$
}$\left(\mathscr{E}^{+}_{\mathcal{H}}\right)^{4}$ is stable under
the operator\textup{ $\mathcal{T}^{\sigma}_{\tau,\tau'}$ for sufficiently
large $\sigma.$ } 
\end{prop}

\begin{proof}
Assume that $M=(M_{1},M_{2},M_{3},M_{4})\in(\mathscr{E}^{+}_{\mathcal{H}})^{4}$.
According to the definition of $\left(\mathscr{E}_{\mathcal{H}}\right)^{4}$
given in Notation \ref{kqqkkd}, the partial derivatives $\dfrac{\partial M_{i}}{\partial t},\dfrac{\partial M_{i}}{\partial x},\dfrac{\partial M_{i}}{\partial y}$
(for $i=1,2,3,4$) are defined everywhere on $\operatorname{int}(K_{\tau,\tau'})\setminus\bigcup^{12}_{m=1}H_{m}$,
where they are continuous and bounded.

Let $\left(t,x,y\right)\in\operatorname{int}(K_{\tau,\tau'})\setminus{\displaystyle \cup^{12}_{m=1}}H_{m}.$
\\
 Setting $\left(t',x',y'\right)=\left(s,x+c\left(s-t\right)\cos\theta,y+c\left(s-t\right)\sin\theta\right)$
for $s\in\mathbb{R}$, it follows from \eqref{qqqq}-\eqref{zzzzz}
that $\left(t',x',y'\right)\notin H_{1}\cup H_{2}\cup H_{3}$ and
that for each $m\in\left\{ 4,\cdots,12\right\} ,$ $\left(t',x',y'\right)\in H_{m}$
for a unique value of $s$. Since $|M|=M$, the integrands in the
expressions \eqref{eq:oloooso-1-1}--\eqref{eq:olole-1-1-1-1-1-2-2}
possess continuous first-order partial derivatives except at a finite
number of points.

The continuity of both the integrands and their derivatives ensures
that $\dfrac{\partial\left(\mathcal{T}^{\sigma}_{\tau,\tau'}M\right)_{1}}{\partial t},$ $\dfrac{\partial\left(\mathcal{T}^{\sigma}_{\tau,\tau'}M\right)_{1}}{\partial x}$
and $\dfrac{\partial\left(\mathcal{T}^{\sigma}_{\tau,\tau'}M\right)_{1}}{\partial y}$
exist at $(t,x,y)$ and can be computed by differentiating under the
integral sign, which implies that they are continuous and bounded.
An identical approach applied to equations \eqref{eq:ppmms-2-1-3-1-1}-\eqref{eq:olole-1-1-1-1-1-2-2}
yields the same properties for $\dfrac{\partial\left(\mathcal{T}^{\sigma}_{\tau,\tau'}M\right)_{i}}{\partial t},$ $\dfrac{\partial\left(\mathcal{T}^{\sigma}_{\tau,\tau'}M\right)_{i}}{\partial x}$
and $\dfrac{\partial\left(\mathcal{T}^{\sigma}_{\tau,\tau'}M\right)_{i}}{\partial y}$
(for $i=2,3,4$), thereby establishing that $\mathcal{T}^{\sigma}_{\tau,\tau'}M\in(\mathscr{E}_{\mathcal{H}})^{4}$.

Furthermore, $\mathcal{T}^{\sigma}_{\tau,\tau'}M\geq0$ holds for
a sufficiently large $\sigma$ . We conclude that $\mathcal{T}^{\sigma}_{\tau,\tau'}M\in(\mathscr{E}^{+}_{\mathcal{H}})^{4}$. 
\end{proof}

\begin{prop}
\label{prop::opps-1-1} Let $R_{0}>0$ be fixed and suppose that $\tau'-\tau\leq1$.
Then for all positive $R$ such that $R\leq R_{0},$ we have for a
sufficiently large $\sigma$
\begin{equation}
\mathcal{T}^{\sigma}_{\tau,\tau'}\left(\mathscr{M}^{+}_{R}\right)\subset\mathscr{M}^{+}_{\mathfrak{p}_{\sigma}R^{2}+\mathfrak{q}_{\sigma}}.\label{soskki}
\end{equation}
Thus $\mathcal{T}^{\sigma}_{\tau,\tau'}$ is compact on $\mathscr{M}^{+}_{R}$
for all $R,$ $0<R\leq R_{0}.$ 
\end{prop}

\begin{proof}
1) Let $M\in\mathscr{M}^{+}_{R}.$ By virtue of \eqref{eq:olole-1-3},
$\left(\mathcal{T}^{\sigma}_{\tau,\tau'}M\right)^{A}_{1}(t,x,y)$
can be represented as 
\begin{multline}
\left(\mathcal{T}^{\sigma}_{\tau,\tau'}M\right)^{A}_{1}\left(t,x,y\right)=\\
{\displaystyle \int^{t}_{0}}e^{-\sigma\int^{t}_{s}\rho\left(M\right)\left(r,x+c\left(r-t\right)\cos\theta,y+c\left(r-t\right)\sin\theta\right)dr}\cdot Q^{\sigma}_{1}\left(M\right)\\
\left(s,x+c\left(s-t\right)\cos\theta,y+c\left(s-t\right)\sin\theta\right)ds+\\
e^{-\sigma{\textstyle {\displaystyle \int^{t}_{0}}}\rho\left(M\right)\left(s,x+c\left(s-t\right)\cos\theta,y+c\left(s-t\right)\sin\theta\right)ds}\overline{N^{\tau}_{1}}\left(t,x,y\right)\label{eq:uiepom-1-1-1}
\end{multline}
Taking the supremum norm yields $\left\Vert \left(\mathcal{T}^{\sigma}_{\tau,\tau'}M\right)^{A}_{1}\right\Vert _{\infty}\leq\left(\tau'-\tau\right)\left\Vert Q^{\sigma}_{1}\left(M\right)\right\Vert _{\infty}+\left\Vert \overline{N^{\tau}_{1}}\right\Vert _{1}.$
Recalling from Notation \eqref{nnooauu} that $\left\Vert Q^{\sigma}_{i}(M)\right\Vert _{\infty}\leq(4\sigma+4cS)\left(\mathscr{N}(M)\right)^{2}$,
we arrive at the upper bound

\begin{equation}
\left\Vert \left(\mathcal{T}^{\sigma}_{\tau,\tau'}M\right)^{A}_{1}\right\Vert _{\infty}\leq\left(4\sigma+4cS\right)\left(\tau'-\tau\right)\left(\mathscr{N}\left(M\right)\right)^{2}+\left\Vert \overline{N^{\tau}_{1}}\right\Vert _{1}.\label{eq:hqiqok-2}
\end{equation}
2) Next, for $M\in\mathscr{M}^{+}_{R},$ 
\begin{multline}
\dfrac{\partial}{\partial t}\left(\mathcal{T}^{\sigma}_{\tau,\tau'}M\right)^{A}_{1}\left(t,x,y\right)=Q^{\sigma}_{1}\left(M\right)\left(t,x,y\right)+{\displaystyle \int^{t}_{0}}\left\{ \right.-\sigma\left[\right.\rho\left(M\right)\left(t,x,y\right)+\\
\int^{t}_{s}{\textstyle \left(\right.-\cos\theta\dfrac{\partial}{\partial x}\rho\left(M\right)-c\sin\theta\dfrac{\partial}{\partial y}\rho\left(M\right)\left.\right)}\\
\left(r,x+c\left(r-t\right)\cos\theta,y+c\left(r-t\right)\sin\theta\right)dr\left.\right]\\
e^{-\sigma\int^{t}_{s}\rho\left(M\right)\left(r,x+c\left(r-t\right)\cos\theta,y+c\left(r-t\right)\sin\theta\right)dr}\cdot Q^{\sigma}_{1}\left(M\right)\left(s,x+c\left(s-t\right)\cos\theta,y+c\left(s-t\right)\sin\theta\right)+\\
e^{-\sigma\int^{t}_{s}\rho\left(M\right)\left(r,x+c\left(r-t\right)\cos\theta,y+c\left(r-t\right)\sin\theta\right)dr}\cdot\\
{\textstyle \left(\right.-\cos\theta\dfrac{\partial}{\partial x}Q^{\sigma}_{1}\left(M\right)-c\sin\theta\dfrac{\partial}{\partial y}Q^{\sigma}_{1}\left(M\right)\left.\right)}\\
\left(s,x+c\left(s-t\right)\cos\theta,y+c\left(s-t\right)\sin\theta\right)\left.\right\} ds\\
-\sigma\left\{ \right.\rho\left(M\right)\left(t,x,y\right)+\int^{t}_{0}{\textstyle \left(\right.-\cos\theta\dfrac{\partial}{\partial x}\rho\left(M\right)-c\sin\theta\dfrac{\partial}{\partial y}\rho\left(M\right)\left.\right)}\\
\left(s,x+c\left(s-t\right)\cos\theta,y+c\left(s-t\right)\sin\theta\right)ds\left.\right\} \\
e^{-\sigma{\textstyle {\displaystyle \int^{t}_{0}}}\rho\left(M\right)\left(s,x+c\left(s-t\right)\cos\theta,y+c\left(s-t\right)\sin\theta\right)ds}.\overline{N^{\tau}_{1}}\left(t,x,y\right)+\\
e^{-\sigma{\textstyle {\displaystyle \int^{t}_{0}}}\rho\left(M\right)\left(s,x+c\left(s-t\right)\cos\theta,y+c\left(s-t\right)\sin\theta\right)ds}\cdot\dfrac{\partial}{\partial t}\overline{N^{\tau}_{1}}\left(t,x,y\right)\label{eq:lopds}
\end{multline}
and 
\begin{multline}
\dfrac{\partial}{\partial x}\left(\mathcal{T}^{\sigma}_{\tau,\tau'}M\right)^{A}_{1}\left(t,x,y\right)={\displaystyle \int^{t}_{0}}\left\{ \right.-\sigma\int^{t}_{s}{\textstyle \dfrac{\partial}{\partial x}\rho\left(M\right)}\left(r,x+c\left(r-t\right)\cos\theta,y+c\left(r-t\right)\sin\theta\right)dr\\
e^{-\sigma\int^{t}_{s}\rho\left(M\right)\left(r,x+c\left(r-t\right)\cos\theta,y+c\left(r-t\right)\sin\theta\right)dr}\cdot Q^{\sigma}_{1}\left(M\right)\left(s,x+c\left(s-t\right)\cos\theta,y+c\left(s-t\right)\sin\theta\right)+\\
e^{-\sigma\int^{t}_{s}\rho\left(M\right)\left(r,x+c\left(r-t\right)\cos\theta,y+c\left(r-t\right)\sin\theta\right)dr}\cdot\\
{\textstyle \dfrac{\partial}{\partial x}Q^{\sigma}_{1}\left(M\right)}\left(s,x+c\left(s-t\right)\cos\theta,y+c\left(s-t\right)\sin\theta\right)\left.\right\} ds\\
-\sigma\left\{ \right.\int^{t}_{0}{\textstyle \dfrac{\partial}{\partial x}\rho\left(M\right)}\left(s,x+c\left(s-t\right)\cos\theta,y+c\left(s-t\right)\sin\theta\right)ds\left.\right\} \\
e^{-\sigma{\textstyle {\displaystyle \int^{t}_{0}}}\rho\left(M\right)\left(s,x+c\left(s-t\right)\cos\theta,y+c\left(s-t\right)\sin\theta\right)ds}.\overline{N^{\tau}_{1}}\left(t,x,y\right)+\\
e^{-\sigma{\textstyle {\displaystyle \int^{t}_{0}}}\rho\left(M\right)\left(s,x+c\left(s-t\right)\cos\theta,y+c\left(s-t\right)\sin\theta\right)ds}\cdot\dfrac{\partial}{\partial x}\overline{N^{\tau}_{1}}\left(t,x,y\right)\label{eq:lopss}
\end{multline}
Using the bounds\\
 $\left\Vert Q^{\sigma}_{i}\left(M\right)\right\Vert _{\infty}\leq\left(4\sigma+4cS\right)\left(\mathscr{N}\left(M\right)\right)^{2},$$\left\Vert \nabla_{t,x,y}Q^{\sigma}_{i}(M)\right\Vert \leq\left(8\sigma+8cS\right)\left(\mathscr{N}\left(M\right)\right)^{2},$\\
 $\left\Vert \rho\left(M\right)\right\Vert _{\infty}\leq4\mathscr{N}\left(M\right)$
and $\left\Vert \nabla_{t,x,y}\rho\left(M\right)\right\Vert \leq4\mathscr{N}\left(M\right)$
, together with the constraints $\mathscr{N}(M)\leq R_{0}$ and $\tau'-\tau\leq1$,
we establish the following estimates: 
\begin{multline}
\left\Vert \dfrac{\partial}{\partial t}\left(\mathcal{T}^{\sigma}_{\tau,\tau'}M\right)^{A}_{1}\right\Vert _{\infty}\leq\left[\right.4+8c\cos\theta+8c\sin\theta+\\
\left(\right.16+16c\cos\theta+16c\sin\theta\left.\right)\sigma R_{0}\left(\tau'-\tau\right)\left.\right]\left(\sigma+cS\right)\left(\mathscr{N}\left(M\right)\right)^{2}+\\
\left[\right.1+\left(\right.4+4c\cos\theta+4c\sin\theta\left.\right)\sigma R_{0}\left.\right]\left\Vert \overline{N^{\tau}_{1}}\right\Vert _{1}\label{eq:oookkz-2}
\end{multline}
\begin{multline}
\left\Vert \dfrac{\partial}{\partial x}\left(\mathcal{T}^{\sigma}_{\tau,\tau'}M\right)^{A}_{1}\right\Vert _{\infty}\leq\left[\right.8+16\sigma R_{0}\left(\tau'-\tau\right)\left.\right]\left(\sigma+cS\right)\left(\mathscr{N}\left(M\right)\right)^{2}+\\
\left[\right.1+4\sigma R_{0}\left.\right]\left\Vert \overline{N^{\tau}_{1}}\right\Vert _{1}.\label{eq:oookkz-2-1}
\end{multline}
By an analogous argument, 
\begin{multline}
\left\Vert \dfrac{\partial}{\partial y}\left(\mathcal{T}^{\sigma}_{\tau,\tau'}M\right)^{A}_{1}\right\Vert _{\infty}\leq\left[\right.8+16\sigma R_{0}\left(\tau'-\tau\right)\left.\right]\left(\sigma+cS\right)\left(\mathscr{N}\left(M\right)\right)^{2}+\\
\left[\right.1+4\sigma R_{0}\left.\right]\left\Vert \overline{N^{\tau}_{1}}\right\Vert _{1}.\label{eq:oookkz-2-1-1}
\end{multline}
Furthermore, invoking \eqref{eq:hqiqok-2} and \eqref{eq:oookkz-2}--\eqref{eq:oookkz-2-1-1},
we have 
\begin{multline}
\left\Vert \left(\mathcal{T}^{\sigma}_{\tau,\tau'}M\right)^{A}_{1}\right\Vert _{1}\leq\left[\right.8+8c\cos\theta+8c\sin\theta+\\
\left(\right.16+16c\cos\theta+16c\sin\theta\left.\right)\sigma R_{0}\left(\tau'-\tau\right)\left.\right]\left(\sigma+cS\right)\left(\mathscr{N}\left(M\right)\right)^{2}+\\
\left[\right.1+\left(\right.4+4c\cos\theta+4c\sin\theta\left.\right)\sigma R_{0}\left.\right]\left\Vert \overline{N^{\tau}_{1}}\right\Vert _{1}.\label{eq:oookkz-2-2}
\end{multline}
Exploiting the structural similarity among the expressions for $\left(\mathcal{T}^{\sigma}_{\tau,\tau'}M\right)^{A}_{i}$
(cf. \eqref{eq:olole-1-3}, \eqref{eq:olole-1-1-1-1-1-1-1}, \eqref{eq:olole-1-2-2},
and \eqref{eq:olole-1-2-1-2}), we conclude that for each $i\in\{1,2,3,4\}$:
\begin{multline}
\left\Vert \left(\mathcal{T}^{\sigma}_{\tau,\tau'}M\right)^{A}_{i}\right\Vert _{1}\leq\left[\right.8+8c\cos\theta+8c\sin\theta+\\
\left(\right.16+16c\cos\theta+16c\sin\theta\left.\right)\sigma R_{0}\left(\tau'-\tau\right)\left.\right]\left(\sigma+cS\right)\left(\mathscr{N}\left(M\right)\right)^{2}+\\
\left[\right.1+\left(\right.4+4c\cos\theta+4c\sin\theta\left.\right)\sigma R_{0}\left.\right]\left\Vert \overline{N^{\tau}_{i}}\right\Vert _{1}.\label{eq:oookkz-2-2-1}
\end{multline}
3) Following a similar line of reasoning, it follows from \eqref{eq:olole-1-1-2}
that 
\begin{equation}
\left\Vert \left(\mathcal{T}^{\sigma}_{\tau,\tau'}M\right)^{B}_{1}\right\Vert _{\infty}\leq\left(4\sigma+4cS\right)\left(\tau'-\tau\right)\left(\mathscr{N}\left(M\right)\right)^{2}+\left\Vert \overline{N^{-}_{1}}\right\Vert _{1}.\label{eq:omplo}
\end{equation}
\begin{multline}
\left\Vert \dfrac{\partial}{\partial t}\left(\mathcal{T}^{\sigma}_{\tau,\tau'}M\right)^{B}_{1}\right\Vert _{\infty}\leq\left[\right.8+8c\cos\theta+8c\sin\theta+\\
\left(\right.32+16c\cos\theta+16c\sin\theta\left.\right)\sigma R_{0}\left(\tau'-\tau\right)\left.\right]\left(\sigma+cS\right)\left(\mathscr{N}\left(M\right)\right)^{2}+\\
\left[\right.1+\left(\right.8+4c\cos\theta+4c\sin\theta\left.\right)\sigma R_{0}\left.\right]\left\Vert \overline{N^{-}_{1}}\right\Vert _{1}\label{eq:oookkz-2-3}
\end{multline}
\begin{multline}
\left\Vert \dfrac{\partial}{\partial x}\left(\mathcal{T}^{\sigma}_{\tau,\tau'}M\right)^{B}_{1}\right\Vert _{\infty}\leq\left[\right.\dfrac{4}{c\cos\theta}+8+16\sigma R_{0}\left(\tau'-\tau\right)\left.\right]\left(\sigma+cS\right)\left(\mathscr{N}\left(M\right)\right)^{2}+\\
\left[\right.\dfrac{4\sigma}{c\cos\theta}R_{0}+4\sigma R_{0}+1\left.\right]\left\Vert \overline{N^{-}_{1}}\right\Vert _{1}\label{eq:oookkz-2-1-2}
\end{multline}
\begin{multline}
\left\Vert \dfrac{\partial}{\partial y}\left(\mathcal{T}^{\sigma}_{\tau,\tau'}M\right)^{B}_{1}\right\Vert _{\infty}\leq\left[\right.8+16\sigma R_{0}\left(\tau'-\tau\right)\left.\right]\left(\sigma+cS\right)\left(\mathscr{N}\left(M\right)\right)^{2}+\\
\left[\right.4\sigma R_{0}+1\left.\right]\left\Vert \overline{N^{-}_{1}}\right\Vert _{1}.\label{eq:oookkz-2-1-1-1}
\end{multline}
It follows from the inequalities \eqref{eq:omplo} -\eqref{eq:oookkz-2-1-1-1}
that 
\begin{multline}
\left\Vert \left(\mathcal{T}^{\sigma}_{\tau,\tau'}M\right)^{B}_{1}\right\Vert _{1}\leq\left[\right.8+\dfrac{4}{c\cos\theta}+8c\cos\theta+8c\sin\theta+\\
\left(\right.32+16c\cos\theta+16c\sin\theta\left.\right)\sigma R_{0}\left(\tau'-\tau\right)\left.\right]\left(\sigma+cS\right)\left(\mathscr{N}\left(M\right)\right)^{2}+\\
\left[\right.1+\left(\right.8+\dfrac{4}{c\cos\theta}+4c\cos\theta+4c\sin\theta\left.\right)\sigma R_{0}\left.\right]\left\Vert \overline{N^{-}_{1}}\right\Vert _{1}.\label{eq:oookkz-2-2-2}
\end{multline}
4) Due to the structural similarity between the expressions for $\left(\mathcal{T}^{\sigma}_{\tau,\tau'}M\right)^{B}_{1}$
and $\left(\mathcal{T}^{\sigma}_{\tau,\tau'}M\right)^{C}_{1}$ (see
Eqs. \eqref{eq:olole-1-1-2}, \eqref{eq:olole-1-1-1-1-3}) we have
\begin{multline}
\left\Vert \left(\mathcal{T}^{\sigma}_{\tau,\tau'}M\right)^{C}_{1}\right\Vert _{1}\leq\left[\right.8+\dfrac{4}{c\sin\theta}+8c\cos\theta+8c\sin\theta+\\
\left(\right.32+16c\cos\theta+16c\sin\theta\left.\right)\sigma R_{0}\left(\tau'-\tau\right)\left.\right]\left(\sigma+cS\right)\left(\mathscr{N}\left(M\right)\right)^{2}+\\
\left[\right.1+\left(\right.8+\dfrac{4}{c\sin\theta}+4c\cos\theta+4c\sin\theta\left.\right)\sigma R_{0}\left.\right]\left\Vert \overline{N^{--}_{1}}\right\Vert _{1}.\label{eq:oookkz-2-2-2-1}
\end{multline}
Combining Eqs. \eqref{eq:oookkz-2-2}, \eqref{eq:oookkz-2-2-2} and
\eqref{eq:oookkz-2-2-2}, we obtain 
\begin{multline}
\left\Vert \left(\mathcal{T}^{\sigma}_{\tau,\tau'}M\right)_{1}\right\Vert _{1}\leq\left[\right.8+\dfrac{4}{c\cos\theta}+\dfrac{4}{c\sin\theta}+8c\cos\theta+8c\sin\theta+\\
\left(\right.32+16c\cos\theta+16c\sin\theta\left.\right)\sigma R_{0}\left(\tau'-\tau\right)\left.\right]\left(\sigma+cS\right)\left(\mathscr{N}\left(M\right)\right)^{2}+\\
\left[\right.1+\left(\right.8+\dfrac{4}{c\cos\theta}+\dfrac{4}{c\sin\theta}+4c\cos\theta+4c\sin\theta\left.\right)\sigma R_{0}\left.\right]\\
\max\left\{ \right.\left\Vert \overline{N^{\tau}_{1}}\right\Vert _{1},\left\Vert \overline{N^{-}_{1}}\right\Vert _{1},\left\Vert \overline{N^{--}_{1}}\right\Vert _{1}\left.\right\} .\label{eq:oookkz-2-2-2-2}
\end{multline}
By analogy, considering the formulations of$\left(\mathcal{T}^{\sigma}_{\tau,\tau'}M\right)_{i},$
$\left(i=1,2,3,4\right)$ (cf. Eqs. \eqref{eq:oloooso-1-1} through
\eqref{eq:olole-1-1-1-1-1-2-2}), it follows that 
\begin{multline}
\mathscr{N}\left(\mathcal{T}^{\sigma}_{\tau,\tau'}M\right)\leq\left[\right.8+\dfrac{4}{c\cos\theta}+\dfrac{4}{c\sin\theta}+8c\cos\theta+8c\sin\theta+\\
\left(\right.32+16c\cos\theta+16c\sin\theta\left.\right)\sigma R_{0}\left(\tau'-\tau\right)\left.\right]\left(\sigma+cS\right)\left(\mathscr{N}\left(M\right)\right)^{2}+\\
\left[\right.1+\left(\right.8+\dfrac{4}{c\cos\theta}+\dfrac{4}{c\sin\theta}+4c\cos\theta+4c\sin\theta\left.\right)\sigma R_{0}\left.\right]\cdot\\
\max_{1\leq i\leq4}\Biggl\{\left\Vert \overline{N^{\tau}_{i}}\right\Vert _{1},\left\Vert \overline{N^{-}_{1}}\right\Vert _{1},\left\Vert \overline{N^{--}_{1}}\right\Vert _{1},\left\Vert \overline{N^{+}_{2}}\right\Vert _{1},\\
\left\Vert \overline{N^{--}_{2}}\right\Vert _{1},\left\Vert \overline{N^{-}_{3}}\right\Vert _{1},\left\Vert \overline{N^{++}_{3}}\right\Vert _{1},\left\Vert \overline{N^{+}_{4}}\right\Vert _{1},\left\Vert \overline{N^{++}_{4}}\right\Vert _{1}\Biggr\}.\label{eq:oookkz-2-2-2-2-1}
\end{multline}
Consequently, for any $\ensuremath{M\in\mathscr{M}^{+}_{R}},$ we
can write 
\begin{equation}
\mathscr{N}\left(\mathcal{T}^{\sigma}_{\tau,\tau'}M\right)\leq\mathfrak{p}_{\sigma}R^{2}+\mathfrak{q}_{\sigma}\label{eq:opoql-1}
\end{equation}
where the coefficients$\ensuremath{\mathfrak{p}_{\sigma}}$ and $\ensuremath{\mathfrak{q}_{\sigma}}$
are defined as in \ref{rrrrt} 
\begin{multline}
p_{\sigma}\equiv\left[\right.8+\dfrac{4}{c\cos\theta}+\dfrac{4}{c\sin\theta}+8c\cos\theta+8c\sin\theta+\\
\left(\right.32+16c\cos\theta+16c\sin\theta\left.\right)\sigma R_{0}\left(\tau'-\tau\right)\left.\right]\left(\sigma+cS\right)\label{eq:olqpq}
\end{multline}
and 
\begin{multline}
q_{\sigma}\equiv\left[\right.1+\left(\right.8+\dfrac{4}{c\cos\theta}+\dfrac{4}{c\sin\theta}+4c\cos\theta+4c\sin\theta\left.\right)\sigma R_{0}\left.\right]\cdot\\
\max_{1\leq i\leq4}\Biggl\{\left\Vert \overline{N^{\tau}_{i}}\right\Vert _{1},\left\Vert \overline{N^{-}_{1}}\right\Vert _{1},\left\Vert \overline{N^{--}_{1}}\right\Vert _{1},\left\Vert \overline{N^{+}_{2}}\right\Vert _{1},\\
\left\Vert \overline{N^{--}_{2}}\right\Vert _{1},\left\Vert \overline{N^{-}_{3}}\right\Vert _{1},\left\Vert \overline{N^{++}_{3}}\right\Vert _{1},\left\Vert \overline{N^{+}_{4}}\right\Vert _{1},\left\Vert \overline{N^{++}_{4}}\right\Vert _{1}\Biggr\}.\label{eq:qoplq}
\end{multline}
5) Since $\mathcal{T}^{\sigma}_{\tau,\tau'}\left(\mathscr{M}^{+}_{R}\right)\subset\left(\mathscr{E}^{+}_{\mathcal{H}}\right)^{4},$
Eq. \eqref{eq:opoql-1} implies that $\mathcal{T}^{\sigma}_{\tau,\tau'}\left(\mathscr{M}^{+}_{R}\right)\subset\mathscr{M}^{+}_{\mathfrak{p}_{\sigma}R^{2}+\mathfrak{q}_{\sigma}}.$
Since $\mathscr{M}^{+}_{\mathfrak{p}_{\sigma}R^{2}+\mathfrak{q}_{\sigma}}$
is relatively compact in $\left(C\left(\begin{array}{r}
K_{\tau,\tau'}\end{array};\R\right),\left\Vert \cdot\right\Vert \right)^{4},$ we conclude that $\mathcal{T}^{\sigma}_{\tau,\tau'}$ is compact
on $\mathscr{M}^{+}_{R}.$ 
\end{proof}

\section{Uniqueness, local solution and global extension }\label{sec:Local-solution}

To establish the uniqueness theorem for the current problem, we recall
that the estimate \eqref{eq:lsoopa-2-1} for the operator $\mathcal{T}_{\tau,\tau'}$,
combined with property (i) in Proposition \ref{aalalaooa}, are the
only tools required to prove uniqueness for the analogous problem
in \cite{sob almeida 2026 arxiv 1} (see in \cite{sob almeida 2026 arxiv 1}
, Theorem 2.1 and its proof in Section 5). Consequently, the uniqueness
theorem remains formally valid for the problem considered herein.

By applying the notations introduced in \eqref{ooodl}, \eqref{ppmmzp},
\eqref{eq:loi-1}, \eqref{kqiqiiq} and \eqref{sskks}, together with
the inclusion \eqref{soskki}, we recover the exact formalism used
in Sections 8 and 9 of \cite{sob almeida 2026 arxiv 1} to establish
both the local solution and its global extension. Specifically, let
$R_{0}>0$ be fixed, and following Section 7 in \cite{sob almeida 2026 arxiv 1},
we have: 
\begin{prop}
\label{izujuu}For a $\sigma$ sufficiently large, assume that $\mathfrak{q}<\dfrac{1}{4\mu\left(1+\delta\sigma R_{0}\right)\left(\sigma+cS\right)}$
and $\tau'-\tau\leq\min\left\{ 1;\dfrac{1}{\lambda\sigma R_{0}}\left(\dfrac{1}{4\mathfrak{q}_{\sigma}\left(\sigma+cS\right)}-\mu\right)\right\} ;$
then $\mathfrak{p}_{\sigma}\mathfrak{q}_{\sigma}\leq\dfrac{1}{4}.$
Under this assumption, for 

\begin{equation}
R=\dfrac{1-\sqrt{1-4\mathfrak{p}_{\sigma}\mathfrak{q}_{\sigma}}}{2\mathfrak{p}_{\sigma}}\label{eq:asloz-1-1-2}
\end{equation}
 $\mathscr{M}^{+}_{R}$ is stable under $\mathcal{T}^{\sigma}_{\tau,\tau'}$. 
\end{prop}

Following Section 8 in \cite{sob almeida 2026 arxiv 1}, this leads
to the first part of Theorem \ref{thm:Suppose-.-Then-1-1-1} for the
local solution. Then, considering the expression of $\gamma$ \eqref{xwccv},
we obtain the following inequalities for any $\tau\geq0:$ \\
 $\left\Vert \overline{N^{\tau}_{1}}\right\Vert _{1}\leq\gamma\left\Vert N^{\tau}_{1}\right\Vert _{1},$$\left\Vert \overline{N^{-}_{1}}\right\Vert _{1}\leq\gamma\left\Vert N^{-}_{1}\right\Vert _{1},$$\left\Vert \overline{N^{--}_{1}}\right\Vert _{1}\leq\gamma\left\Vert N^{--}_{1}\right\Vert _{1},$\\
 $\left\Vert \overline{N^{\tau}_{2}}\right\Vert _{1}\leq\gamma\left\Vert N^{\tau}_{2}\right\Vert _{1},$$\left\Vert \overline{N^{+}_{2}}\right\Vert _{1}\leq\gamma\left\Vert N^{+}_{2}\right\Vert _{1},$$\left\Vert \overline{N^{--}_{2}}\right\Vert _{1}\leq\gamma\left\Vert N^{--}_{2}\right\Vert _{1},$\\
 $\left\Vert \overline{N^{\tau}_{3}}\right\Vert _{1}\leq\gamma\left\Vert N^{\tau}_{3}\right\Vert _{1},$$\left\Vert \overline{N^{-}_{3}}\right\Vert _{1}\leq\gamma\left\Vert N^{-}_{3}\right\Vert _{1},$$\left\Vert \overline{N^{++}_{3}}\right\Vert _{1}\leq\gamma\left\Vert N^{++}_{3}\right\Vert _{1},$\\
 $\left\Vert \overline{N^{\tau}_{4}}\right\Vert _{1}\leq\gamma\left\Vert N^{\tau}_{4}\right\Vert _{1},$$\left\Vert \overline{N^{+}_{4}}\right\Vert _{1}\leq\gamma\left\Vert N^{+}_{4}\right\Vert _{1},$$\left\Vert \overline{N^{++}_{4}}\right\Vert _{1}\leq\gamma\left\Vert N^{++}_{4}\right\Vert _{1}$.\\
 Using the definition of the parameter $\mathfrak{q}$ \eqref{eq:loso-1-1-2},
it follows that 
\begin{multline*}
\mathfrak{q}<\gamma\cdot{\displaystyle \max_{1\leq i\leq4}\left\{ \right.\left\Vert N^{\tau}_{i}\right\Vert _{1},\left\Vert N^{-}_{1}\right\Vert _{1},\left\Vert N^{--}_{1}\right\Vert _{1},\left\Vert N^{+}_{2}\right\Vert _{1},\left\Vert N^{--}_{2}\right\Vert _{1},}\\
\left\Vert N^{-}_{3}\right\Vert _{1},\left\Vert N^{++}_{3}\right\Vert _{1},\left\Vert N^{+}_{4}\right\Vert _{1},\left\Vert N^{++}_{4}\right\Vert _{1}\left.\right\} .
\end{multline*}
This is used to prove the second part of Theorem \ref{thm:Suppose-.-Then-1-1-1},
which corresponds exactly to Corollary 8.3 in \cite{sob almeida 2026 arxiv 1}.

Theorem \ref{thm:Suppose-.-Then-1-2} for the global solution follows
thereafter, as the proof falls into the same formalism as the global
extension in \cite{sob almeida 2026 arxiv 1} (Theorem 2.3 and its
proof in Section 9).

\section*{Conclusion}

In this paper, we have considered the initial boundary value problem
for the general Broadwell four-velocity planar model. We have successfully
defined parameters analogous to those in \cite{sob almeida 2026 arxiv 1},
enabling the direct applicability of the formalism from \cite{sob almeida 2026 arxiv 1}
to prove the existence of a unique local classical solution, as well
as a unique global solution achieved through a precise threshold condition
imposed on the uniform norms of the data and their derivatives.

Crucially, the success of this approach under generalized geometric
conditions validates our initial objective: creating a flexible and
universal analytical tool for discrete kinetic theory. Consequently,
future research will focus on extending this fixed-point and threshold
methodology to abstract discrete velocity models.

\end{document}